\documentclass[11pt]{amsart}
\usepackage{amsmath}
\usepackage{amssymb}
\usepackage{amscd}
\usepackage[utf8]{inputenc}
\usepackage{graphicx}

\newtheorem{thm}{\sc theorem}[section]
\newtheorem{dfn}[thm]{\sc definition}
\newtheorem{pro}[thm]{\sc proposition}
\newtheorem{lem}[thm]{\sc lemma}
\newtheorem{cor}[thm]{\sc corollary}

\newtheorem*{TA}{\sc theorem A}
\newtheorem*{TC}{\sc theorem C}
\newtheorem*{TB}{\sc theorem B}

\newcommand{\F}{\mbox{$\mathcal F$}}

\newcommand{\X}{\mbox{$\mathcal X$}}

\newcommand{\rank}{\mbox{$\mbox{\rm rank}$}}

\def\Z{{\bf Z}}
\def\Q{{\bf Q}}
\def\R{{\bf R}}

\def\S{{\bf S}}

\def\D{{\bf D}}

\def\mun{{^{-1}}}

\def\pr{{\rm pr}}

\def\demi{{2\mun}}

\begin{document}

\title{Realizing compactly generated pseudo-groups of dimension one}
\begin{abstract} Many compactly generated pseudo-groups of local
transformations on $1$-manifolds are realizable as the transverse dynamic of a foliation
of codimension $1$ on a compact
manifold of dimension $3$ or $4$.
\end{abstract}

\maketitle

\begin{center}
Ga\"{e}l Meigniez
\footnote{Partially supported by the Japan Society for Promoting Science,
No.\ L03514}
\end{center}
\date{\today}

After C. Ehresmann,
\cite{ehr1954},
 given a foliation $\F$ of codimension $q$ on a
compact
manifold $M$~, its transverse dynamic
 is represented by its \emph{holonomy
pseudo-group} of local transformations on any exhaustive transversal $T$~.
The inverse problem has been raised by A. Haefliger: given a pseudo-group
 of local transformations of some manifold of dimension $q$~,
realize it, if possible, as the dynamic of some foliation
of codimension $q$ on some compact manifold.
The difficulty here lies in the compactness.
More precisely, Haefliger discovered a necessary condition: the
pseudo-group must be \emph{compactly generated} \cite{hae1985}\cite{hae2002}.
 He asked if this
condition is sufficient.

The present paper intends to study the case $q=1$~.

A counterexample is known: there exists a compactly generated pseudo-group of
local transformations of the line, which is not
realizable. It contains a \emph{paradoxical Reeb component:} a
full subpseudo-group equivalent to the holonomy of a Reeb component,
but whose boundary orbit has some complicated isotropy group on the exterior side
\cite{mei2010}.

The object of the
 present paper is, on the contrary, to give a \emph{positive} answer to Haefliger's
question for many pseudo-groups of dimension one.

 Recall that a codimension $1$ foliation is (topologically)
\emph{taut} if through every point there passes a transverse loop, or a transverse
path with extremities on $\partial M$ (we refer e.g. to \cite{can2000} for the elements on foliations)
. Equivalently, the foliation has no \emph{dead end component.} These notions are easily translated for
pseudo-groups: one has the notions of tautness
and of dead end components for a pseudo-group of dimension $1$ (see paragraph
\ref{tautness_sbs} below). For example, every pseudo-group of dimension $1$
without closed orbit is taut.
It turns out that, to realize a given compactly generated pseudo-group
 of dimension $1$,
the extra necessary and/or sufficient conditions that we find, bear on the isotropy groups
of the closed orbits bounding the dead end components, if any.

We also pay attention to the dimension of the realization. Of course,
 a pseudo-group which is realized by some foliation $\F$ on some manifold $M$~,
 is also realized by the pullback of $\F$ into $M\times\S^1$~. One can
ask to realize a pseudo-group, if possible, in the smallest possible dimension.
The dynamics of the foliations on surfaces being very restrictive,
the dimension $3$ will be in general the first candidate.

There is a well-known constraint specific to dimension $3$ in the nontaut case. Namely,
remember that for elementary
Euler characteristic reasons, in every compact foliated $3$-manifold
which is not taut, every leaf bounding a dead end component is a $2$-torus or an annulus 
(S. Goodman) \cite{can2000}.
 This phenomenon has a counterpart
 in the holonomy pseudo-group: for every
 orbit bounding a dead end component, its isotropy group is commutative of
 rank at most two.

\medbreak
A few precisions must be given before the results.

Equivalence: the pseudo-groups must be considered up to an equivalence
called \emph{Haefliger equivalence}. Given
 two different
exhaustive transversals for a same foliated manifold, the two holonomy pseudo-groups are Haefliger-equivalent \cite{hae1985}\cite{hae1988}\cite{hae2002}.
 A foliated manifold
is said to \emph{realize} a
pseudo-group $G$ if its holonomy pseudo-group on any exhaustive transversal
is Haefliger-equivalent to $G$~.

Differentiability: The given pseudo-group being of class $C^r$~, $0\le r\le\infty$~,
the realizing foliation will be $C^{\infty, r}$, that is, globally $C^r$ and tangentially smooth \cite{can2000}.
We also consider the  pseudo-groups of class $PL$~:
the realizing foliations will be $C^{\infty, PL}$.

Orientation:
for simplicity, all pseudo-groups are understood  \emph{orientable,} that is, orientation-preserving.
All foliations are understood tangentially orientable and transversely orientable.

Boundaries:
by a "foliated manifold", we understand a manifold $M$ with a smooth boundary
(maybe empty), endowed with a foliation $\F$ such that each
connected
component of $\partial M$ is either a leaf of $\F$ or transverse to $\F$~.
 So, $\partial M$ splits into a
tangential boundary $\partial_\parallel M$~, which is seen in the holonomy pseudo-group,
and a transverse boundary $\partial_\pitchfork M$~, which is not.
However, in the realization problem,
the choice of allowing a transverse boundary or not, only affects the dimension of the realization. For,
if $G$ is realized by some foliation $\F$ on some manifold $M$~, with some transverse boundary components, then it is also realized, without transverse boundary components,  by the pullback of $\F$ in
a manifold of one more dimension,
namely:
$$(M\times\S^1)\cup_{(\partial_\pitchfork M\times\S^1)}
(\partial_\pitchfork M\times
\D^2)$$

\begin{TA} Every
pseudo-group of dimension $1$ which is \emph{compactly generated}
and \emph{taut},
is realized by some foliated compact $3$-manifold, without transverse boundary.
\end{TA}

Essentially,  our method is that the pseudo-group is first easily
realized as the dynamic of a Morse-singular foliation on a compact $3$-manifold.
The singularities are of Morse indices $1$ and $2$~, in equal number.
Then, thanks to tautness, from every singularity of index $2$ there is a
positively transverse path to some (distant) singularity of index $1$~. Thanks
  to some geometric manifestation of
   compact generation, the pair is cancelled, \emph{not} in Morse's way, but rather
by the means of an elementary surgery of index $2$
performed on the manifold, without changing the dynamic of the foliation.

An analogous construction can also be made inside a given foliated manifold; this leads
to the following dimension reduction.

\begin{pro}\label{reduction_pro} Let $(M,\F)$ be a compact manifold of dimension
$n\ge 4$, endowed with a foliation of codimension $1$ which is
topologically taut.

Then, there is a proper compact
 submanifold $M'\subset M$ of dimension $n-1$
transverse to $\F$, such that:
\begin{itemize}
\item Every leaf of $\F$ meets $M'$;
\item Every two points of $M'$ which lie on the same leaf of $\F$,
 lie on the same leaf of $\F\vert M'$.
\end{itemize}
\end{pro}

In particular, the holonomy pseudo-group of $\F\vert M'$
is Haefliger-equivalent to the holonomy pseudo-group of $\F$. Here, \emph{proper} means that $\partial_\pitchfork M'=\emptyset$
and that $\partial_\parallel M'=M'\cap\partial_\parallel M$. Of course, from proposition \ref{reduction_pro}, it follows by induction on $n$, that $M$ 
contains  a proper compact
 submanifold \emph{of dimension $3$}
transverse to $\F$, with the same two properties.

Recently, Martinez Torres, Del Pino and Presas have obtained by very different means
a similar result in the particular case where $M$
admits a global closed $2$-form inducing a symplectic form on every leaf \cite{mar2014}. I thank Fran Presas for pointing out to me the general problem.
\medbreak

 We also get
a \emph{characterization} of the dynamics of all foliations, taut or not, on
compact $3$-manifolds.

\begin{TB}\label{B_thm} Let $G$ be a pseudo-group of dimension $1$~. Then,
the following two properties are equivalent.
\begin{enumerate}
\item $G$
is realizable by some foliated compact $3$-manifold
 (possibly with a transverse boundary);

\item $G$ is compactly generated;
and for every orbit of $G$ in the boundary of every dead end component, its isotropy
group is commutative of rank at most $2$~.
\end{enumerate}
\end{TB}

As a basic but fundamental
 example, the nontaut pseudo-group of local transformations of the real
line generated by two homotheties $t\mapsto\lambda t$~,
$t\mapsto\mu t$~, with $\lambda$, $\mu>0$ and $\log \mu/\log\lambda\notin\Q$~,
verifies the properties of theorem B, and is \emph{not}
realizable by any foliated compact $3$-manifold without boundary  --- see
paragraph \ref{homothety_sbs} below. We know no simple, necessary and sufficient
 conditions for realizing a nontaut pseudo-group
on a compact $3$-manifold without transverse boundary.

More generally, skipping the condition of rank at most $2$:

\begin{TC} Let $G$ be a pseudo-group of dimension $1$ which is {compactly generated}
and such that, for every orbit in the boundary of every dead end component,
its isotropy group is commutative. Then, $G$
is realizable by some foliated compact orientable $4$-manifold, without transverse boundary.
\end{TC}

\begin{cor}
Every pseudo-group
of dimension $1$ and of class $PL$  which is compactly generated,
is realizable (in dimension $4$).

\end{cor}

There remain several open questions between these positive results and the negative result of \cite{mei2010}.

Regarding the isotropy groups of the orbits bounding the dead end components,
the zoology of the groups that always allow a realization in high dimension,
remains obscure (and may be intractable).

Consider a pseudo-group $G$ of dimension $1$ which is compactly generated.
If $G$ is real analytic,
is it necessarily realizable? If $G$ is realizable, is it necessarily realizable in dimension $4$~?

Also,
beyond the realization problem itself, one can ask for a more universal property.
Call $G$ \emph{universally realizable}
if there is a system of foliated compact manifolds, each realizing $G$~, and of foliation-preserving embeddings, whose inductive limit is a Haefliger classifying
space for $G$~. One can prove (not tackled in the present paper) that every compactly generated pseudo-group of dimension
one and class $PL$ is
universally realizable.
What if we change $PL$ for "taut"? for "real analytic"?

 The problem is that the method of the present paper, essentially
 the cancellation of a pair of distant singularities
of indices $1$ and $n-1$ in a Morse-singular foliation on a $n$-manifold
by an elementary surgery
without changing the dynamic, is very specific to these indices,
and we know no equivalent
e.g. for a pair of singularities of index $2$ in ambient dimension $4$~.
\medbreak

\section{Preliminaries on pseudo-groups}
In the first two paragraphs \ref{pghe_sbs} and \ref{compact_generation_sbs},
we recall concepts and facts about Haefliger equivalence and compact
generation, in a form that fits our purposes. The material here is essentially due to Haefliger. In the third paragraph \ref{tautness_sbs}, we translate into the frame
of pseudo-groups of dimension $1$, the notion of topological tautness, which is classical in the frame
of foliations of codimension $1$.
\medbreak

In $\R^n$~, one writes
 $\D^n$
the compact unit ball; $\S^{n-1}$ its boundary;
 $*$ the basepoint $(1,0,\dots,0)\in\S^{n-1}$~; and $dx_n$ the foliation
  $x_n=constant$~. ``Smooth'' means $C^\infty$.

\subsection{Pseudo-groups and Haefliger equivalences}\label{pghe_sbs}
 An arbitrary differentiability class is understood. Let $T$~, $T'$ be manifolds of the same dimension, not necessarily compact. Smooth boundaries are allowed.

A \emph{local transformation} from $T$ to $T'$ is a diffeomorphism $\gamma$ between two nonempty, topologically open subsets $Dom(\gamma)\subset T$~,
$Im(\gamma)\subset T'$~.
Note that the boundary is necessarily invariant~:$$Dom(\gamma)\cap\partial T=\gamma\mun(
\partial T'
)$$Given also a local transformation $\gamma'$ from $T'$ to $T''$~,
the composite $\gamma'\gamma$ is defined whenever $Im(\gamma)$ meets $Dom(\gamma')$ (an inclusion is not necessary), and one has~:$$Dom(\gamma'\gamma)=\gamma\mun(Dom(\gamma'))$$
Given two sets of local transformations $A$~, $B$~, as usual, $AB$
denotes the set of the composites of all composable pairs $\alpha\beta$~,
 where $\alpha\in A$ and $\beta\in B$~. Also, $1_U$ denotes the identity map
of the set $U$~.

\begin{dfn}\cite{veb1931}
A \emph{pseudo-group} on a manifold $T$ is a set $G$ of local self-transformations of $T$ such that~:
\begin{enumerate}
\item For every nonempty, topologically open $U\subset T$~, the identity map $1_U$ belongs to $G$~;
\item $GG=G\mun=G$~;
\item For every local self-transformation $\gamma$ of $T$~, if $Dom(\gamma)$ admits an open cover $(U_i)$
 such that every restriction $\gamma\vert U_i$
belongs to $G$~, then $\gamma$ belongs to $G$~.
\end{enumerate}
\end{dfn}

Then, by (1) and (2), $G$ is also stable by restrictions: if $\gamma$ belongs to $G$ and if $U\subset Dom(\gamma)$ is nonempty open, then $\gamma\vert U$ belongs to $G$~.

Example 1. Every set $S$ of local self-transformations of $T$ is contained in a smallest pseudo-group $<S>$ containing $S$~,
called the pseudo-group
\emph{generated} by $S$~. A local transformation $\gamma$ of $T$ belongs to $<S>$ if
and only if, in a neighborhood of every point in its domain, $\gamma$ splits as a composite $\sigma_\ell\dots\sigma_1$~,
with $\ell\ge 0$ and $\sigma_1$~,\dots, $\sigma_\ell\in S\cup S\mun$~.

Example 2. Given a pseudo-group $(G,T)$~, and a nonempty open subset $U\subset T$~,
one has on $U$ a \emph{restricted} pseudo-group $G\vert U:=1_UG1_U$~:
 the set of the elements of $G$ whose domains
and images are both contained in $U$~.

Example 3. More generally, given a pseudo-group $(G,T)$~, a manifold $T'$~, and a set $F$ of local transformations from
$T'$ to $T$~,
 one has on $T'$ a
\emph{pullback} pseudo-group $F^*(G):=<F\mun GF>$~.
\medbreak
Under a pseudo-group $(G,T)$~, every point $t\in T$ has~:
\begin{enumerate}
\item
An \emph{orbit} $G(t)$~: the set of the images $\gamma(t)$ through
 the local transformations
 $\gamma\in G$ defined at $t$~;
\item
An \emph{isotropy group} $G_t$~:
 the group of the germs at $t$ of the
 local transformations $\gamma\in G$ defined at $t$ and fixing $t$~.
\end{enumerate}
 Call an open subset $T'\subset T$ \emph{exhaustive} if $T'$ meets every orbit. Call the pseudo-group  $G$
\emph{cocompact} if $T$ admits a relatively compact exhaustive open subset. Call the pseudo-group $G$
\emph{connected} if every two points of $T$ are linked by a finite sequence of points of $T$~,
 of which every two consecutive ones
lie in the same orbit or in the same connected component of $T$~. Obviously, every pseudo-group splits
as a disjoint sum of connected ones.
\medbreak
Let $(M,\F)$ be a manifold foliated in codimension $q$~. A smooth boundary is allowed, in which case each connected component of $\partial M$ must be tangent to $\F$ or transverse to $\F$~. One writes $\partial_\parallel M$ the union of the tangential components. By a \emph{transversal,} one means a $q$-manifold $T$ immersed into $M$
transversely to $\F$~, not necessarily compact, and such that $\partial T=T\cap\partial_\parallel M$~. One calls $T$ \emph{exhaustive} (or \emph{total}) if it meets every leaf.

\begin{dfn}\label{holonomy_dfn}\cite{ehr1954}
The \emph{holonomy pseudo-group} $Hol(\F,T)$ of a foliation $\F$ on an exhaustive transversal
 $T$ is the pseudo-group \emph{generated} by the local transformations $\gamma$ of $T$
for which there exists a map $$f_\gamma: [0,1]\times Dom(\gamma)\to M$$
such that~:
\begin{itemize}
\item $f_\gamma\pitchfork\F$
and $f_\gamma^*\F$ is the slice foliation on $[0,1]\times Dom(\gamma)$,
whose leaves are the $[0,1]\times t$'s ($t\in Dom(\gamma)$);
\item  $f_\gamma(0,t)=t$ and  $f_\gamma(1,t)=\gamma(t)$~,
 for every $t\in Dom(\gamma)$~.

\end{itemize}
\end{dfn}

We may call $f_\gamma$ a \emph{fence} associated to $\gamma$.
This holonomy pseudo-group does represent the dynamic of the foliation:
 there is a one-to-one correspondence $L\mapsto L\cap T$ between the leaves of $\F$ and  the orbits of $Hol(\F,T)$~; a topologically closed orbit
corresponds to a topologically closed leaf; the isotropy group of $Hol(\F,T)$
at any point
  is isomorphic with the holonomy group of the corresponding leaf; etc.
\begin{dfn}\label{equivalence_dfn}\cite{hae1984}
A \emph{Haefliger equivalence} between two pseudo-groups $(G_i,T_i)$ ($i=0,1$) is a pseudo-group $G$
on the disjoint union $T_0\sqcup T_1$~, such that
 $G\vert T_i=G_i$ ($i=0,1$) and that every orbit of $G$ 
meets both $T_0$ and $T_1$~.
\end{dfn}

Example 1. The two holonomy pseudo-groups of a same foliation on two exhaustive transversals are Haefliger equivalent.

Example 2. The restriction of a pseudo-group $(G,T)$ to any exhaustive open
subset of $T$ is Haefliger-equivalent to $(G,T)$~.

Example 3. More generally, let $(G,T)$ be a pseudo-group, and let $F$ be
 a set of local transformations from $T'$ to $T$~. Assume that:
\begin{enumerate}
\item $FF\mun\subset G$~;
\item $\cup_{\phi\in F}Dom(\phi)=T'$~;
\item $\cup_{\phi\in F}Im(\phi)$ is $G$-exhaustive in $T$~.
\end{enumerate}
Then, the pseudo-group $<F\cup G>$ on $T\sqcup T'$  is a Haefliger equivalence between $(G,T)$ and  $(F^*(G), T')$~.

\medbreak
The Haefliger equivalence is actually an equivalence relation between pseudo-groups.
Given two Haefliger equivalences: $G$
between $(G_0,T_0)$ and $(G_1,T_1)$, and
 $G'$
between $(G_1,T_1)$ and $(G_2,T_2)$, one forms
the pseudo-group $<G\cup G'>$ on $T_0\sqcup T_1\sqcup T_2$~.
Then, $<G\cup G'>\vert(T_0\sqcup T_2)$ is a
 Haefliger equivalence between $(G_0,T_0)$ and $(G_2,T_2)$~.

Every Haefliger equivalence induces a one-to-one correspondence between the orbit spaces
 $T_i/G_i$ ($i=0,1)$~. A closed orbit corresponds to a closed orbit. The isotropy groups at points on corresponding orbits are isomorphic.

\subsection{Compact generation}\label{compact_generation_sbs}

Let $(G,T)$ be a pseudo-group. We say that
 $\gamma\in G$ is  ($G$-) {\it extendable} if there exists some $\bar\gamma\in G$
 such that $Dom(\gamma)$ is contained \emph{and relatively compact}
 in $Dom(\bar\gamma)$, and that $\gamma=\bar\gamma\vert Dom(\gamma)$.
 The composite of two extendable elements is also extendable. The inverse of an extendable element is also extendable.

\begin{dfn}\label{CG_dfn} {\rm (Haefliger)\cite{hae1985}} A pseudo-group $(G,T)$ is \emph{compactly generated} if
 there are an exhaustive, relatively compact, open subset $T'\subset T$~, and finitely many elements
of $G\vert T'$ which are $G$-{extendable}, and which generate $G\vert T'$~.
\end{dfn}

\begin{pro}{\rm (Haefliger)\cite{hae1985}\cite{hae2002}}\label{invariance_pro}
Compact generation is invariant by Haefliger equivalence.
\end{pro}
\begin{pro}{\rm (Haefliger)\cite{hae1985}\cite{hae2002}}
 The holonomy pseudo-group of every foliated compact manifold is compactly generated.
\end{pro}
We shall also use the following fact, which amounts to say that
the choice of $T'$ is arbitrary.
\begin{lem}\label{transversal_change_lem}{\cite{hae1985}\cite{hae2002}}
Let $(G,T)$ be a {compactly generated} pseudo-group,
and $T''\subset T$ be \emph{any}
 exhaustive, relatively compact, open subset.
Then there are finitely many elements
of $G\vert T''$ that are {extendable} in $G$~, and that generate $G\vert T''$~.
\end{lem}

Note: pseudo-groups vs. groupoids. The above definition of compact generation, although it may look strange at first look, is relevant;
 in particular because it is preserved through Haefliger equivalences.
 N. Raimbaud has shown that compact generation has a somewhat more natural generalization
in the frame of topological groupoids.  Write $\Gamma(T)$ the topological
groupoid of the germs of
local transformations of $T$~. Let $G$ be any pseudo-group on $T$~. Then,
the set $\Gamma$ of germs $[g]_t$~, for all $g\in G$ and all $t\in Dom(g)$~, is
in $\Gamma(T)$ an open subgroupoid whose space of objects is the all of $T$~. It is easily verified that
 one gets this way a bijection between the set of pseudo-groups on $T$ and the set of open subgroupoids in $\Gamma(T)$ whose space of objects is $T$~.
 The pseudo-group $G$ is compactly generated if and only if the topological groupoid $\Gamma$
 contains an exhaustive, relatively compact, open subset, which generates a full subgroupoid \cite{rai2009}.

\subsection{Tautness for pseudo-groups of dimension $1$}\label{tautness_sbs}
We now consider a pseudo-group $(G,T)$ \emph{of dimension $1$~,} that is, $\dim T=1$~; and \emph{oriented,} that is, $T$ 
is oriented and $G$ is orientation-preserving. From now on, all pseudo-groups will be understood of dimension $1$ and oriented.

 By a \emph{positive arc} $[t,t']$
of origin $t$ and extremity $t'$~,
we mean an orientation-preserving embedding of the interval $[0,1]$ into $T$ sending $0$ to $t$ and $1$ to $t'$~.

A \emph{positive chain}
is a finite sequence of positive arcs, such that the extremity of each (but the last) lies on the same
orbit as the origin of the next.
A \emph{positive loop} is a positive chain such that the extremity of the last arc lies on the same
orbit at the origin of the first.

\begin{dfn} A pseudo-group $(G,T)$ of dimension 1 is \emph{taut}
if every point of $T$ lies either on a positive chain
whose origin and extremity belong to $\partial T$~, or on  a positive loop.
\end{dfn}
\begin{pro}\label{taut_pro} Let $(G,T)$ be
a cocompact pseudo-group of dimension 1. Then, $(G,T)$
 is taut if and only if it is Haefliger-equivalent to
 some pseudo-group $(G',T')$ such that $T'$ is a finite disjoint
  union of compact intervals and circles.
\end{pro}

\begin{proof} One first easily verifies that tautness is invariant by Haefliger equivalence. "If" follows.

Conversely, given a taut cocompact pseudo-group $(G,T)$~, by cocompactness there
is a finite family $C$ of positive chains, each being
a loop or having extremities on $\partial T$~, such that every orbit of $G$ meets on at least one of them.

Consider one of these chains $c=([t_i,t'_i])$ ($0\le i\le\ell(c)$) which is not a loop:
its origin $t_0$ and extremity $t'_{\ell(c)}$ lie on $\partial T$~.
For every $1\le i\le\ell(c)$~, one has $t_{i}=g_i(t'_{i-1})$ for some $g_i\in G$
whose domain and image are small.
Let $$U_0:=[t_0,t'_0]\cup Dom(g_{1})$$ $$U_{\ell(c)}:=Im(g_{\ell(c)})\cup[t_{\ell(c)},t'_{\ell(c)}]$$ and for each
  $1\le i\le{\ell(c)}-1$~, let $$U_i:=Im(g_i)\cup[t_i,t'_i]\cup Dom(g_{i+1})$$

One makes an abstract copy $U'_i$ of each $U_i$~. Write
 $f_{c,i}:U'_i\to U_i$ for the identity.
 These abstract copies are
glued together by means of the $g_i$'s into a single compact segment $T'_c$~.
Thus, $T'_c$ has an atlas of maps which are
 local transformations $f_{c,i}$ ($0\le i\le\ell(c)$) from $T'_c$ to $T$~, such that every
change of maps $g_i=f_{c,i}f_{c,i-1}\mun$ belongs to $G$~. The images of the $f_{c,i}$'s cover the chain $c$~.

In the same way, for every $c\in C$ which is a loop, one makes a circle
 $T'_c$ together with an atlas $f_{c,i}$ ($0\le i\le\ell(c)$) of maps
 which are local transformations from $T'_c$ to $T$~, such that every
change of maps belongs to $G$~. The images of the maps cover the chain $c$~.

Let $T'$ be the disjoint union of the $T'_c$'s, for $c\in C$~.
By the example 3 above after the definition \ref{equivalence_dfn},
$G$ is Haefliger-equivalent to the pseudo-group $F^*(G)$ of local transformations
of $T'$~.
\end{proof}

In case $(G,T)$ is connected, one can be more precise (left as an exercise):
\begin{pro}\label{connected_taut_pro} Let $(G,T)$ be
a connected, cocompact pseudo-group of dimension 1. Then, $(G,T)$
 is taut if and only if it is Haefliger-equivalent to
 some pseudo-group $(G',T')$ such that $T'$ is either a finite disjoint
  union of compact intervals,
 or a single circle.
\end{pro}

Our last lemma has no relation to tautness. For a compactly generated
pseudo-group of dimension one, one can give a more precise form to
the generating system defining compact generation:

\begin{lem}\label{intervals_lem}
Let $(G,T)$ be a compactly generated pseudo-group of dimension $1$~.
Then~:
\begin{enumerate}
\item There is a $G$-exhaustive, open, relatively compact $T'\subset T$
 which
 has finitely many connected components;

\item For every $T'$ as above, $G\vert T'$
admits a finite set
of $G$-extendable
generators whose domains and images
 are intervals.
\end{enumerate}
\end{lem}
\begin{proof} (1) The pseudo-group $G$~, being compactly generated, is in particular cocompact: there is a compact $K\subset T$ meeting every orbit.
Being compact, $K$ meets only finitely many connected components $T_i$ of $T$~.
For each $i$~, let $T'_i\subset T_i$ be relatively compact, open, connected,
and contain $K\cap T_i$~. Then, $T':=\cup_iT'_i$ is $G$-exhaustive, open,
relatively compact, and has finitely many connected components.

(2) By the lemma \ref{transversal_change_lem}, $G\vert T'$ admits a finite set $(g_i)$ $(i=1,\dots,p)$
of $G$-extendable
generators. For each $1\le i\le p$~, let $\bar g_i$ be a $G$-extension of $g_i$~. Let $U_i\subset Dom(\bar g_i)$ be open, relatively compact, contain $Dom(g_i)$~, and have
finitely many connected components. Then, $U_i\cap T'$ has finitely many connected components. Each of these components is either an interval or a circle.
In the second case, we cover this circle by two open intervals.
We get a cover of $U_i\cap T'$ by a finite family $(I_{j})$ ($j\in J_i$)
of intervals open and relatively compact in $Dom(\bar g_i)$ .
 The finite family $(\bar g_i\vert I_j)$ $(1\le i\le p, j\in J_i)$
is $G$-extendable and generates $G\vert T'$~.
\end{proof}

\section{Proof of theorem A and of proposition \ref{reduction_pro}}\label{A_section}
\def\RR{{\mathcal R}}

\subsection{Proof of theorem A}
We are given a taut, compactly generated pseudo-group
$(G,T)$ of dimension $1$ and class $C^r$, $0\le r\le\infty$, or $PL$ . We have to realize $G$ as the holonomy pseudo-group of some
foliated compact $3$-manifold. By proposition \ref{taut_pro}, we can assume that $T$
is compact: a finite disjoint union of compact intervals and circles.
By lemma \ref{intervals_lem} applied to $T'=T$~, the pseudo-group
$G$ admits a finite system $g_1$~,\dots, $g_p$
 of $G$-extendable generators whose domains and images
are intervals.

The proof uses Morse-singular foliations. It
would be natural to define them as the Haefliger structures whose
singularities are quadratic, but this would lead to irrelevant technicalities.
A simpler concept will do.

\begin{dfn}\label{morse_dfn}
 A \emph{Morse foliation} $\F$ on a smooth $n$-manifold $M$ is a foliation
of codimension one and class $C^{\infty,r}$
on the complement of finitely many singular points, such that on
some open neighborhood of each, $\F$ is
conjugate to the level hypersurfaces of some nondegenerate quadratic form
on some neighborhood of $0$ in $\R^n$~. The conjugation must be $C^0$~;
it must be smooth except maybe at the singular point.
\end{dfn} 

We write $Sing(\F)\subset M$ for the finite set of singularities. Note that $\F$
is smooth on some neighborhood of $Sing(\F)$, minus $Sing(\F)$.
The holonomy pseudo-group of $\F$ is defined,
 on any exhaustive transversal
disjoint from the singularities, as the
holonomy pseudo-group of the regular foliation $\F\vert(M\setminus Sing(\F))$~.

  We shall first realize $(G,T)$ as the holonomy pseudo-group of
a Morse foliation on a compact $3$-manifold. Then, compact generation
will allow us to perform a surgery on this manifold and
regularize the foliation, without changing its transverse structure.

\smallbreak
To fix ideas, at first we assume that $T$ is without boundary: that is, a
finite disjoint union
of circles.

Let $M_0:=\S^2\times T$
and let $\F_0$ be the
foliation of $M_0$ by $2$-spheres: its holonomy pseudo-group on the exhaustive transversal $\ast\times T$
is the trivial pseudo-group. Write $\pr_2:M_0\to T$ for the second projection.

For every $1\le i\le p$~,
write $(u_i,u'_i)\subset T$ for the open interval that is the domain of $g_i$~,
 and write $(v_i,v'_i)$ the image of $g_i$~. Fix some extension $\bar g_i\in G$~.
 
Choose two embeddings $e_i:\D^3\to M_0$ and $f_i:\D^3\to M_0$
such that
\begin{enumerate}
\item $e_i(\D^3)$ and $f_i(\D^3)$ are disjoint from each other and from $\ast\times T$~;
\item $\pr_2(e_i(\D^3))=[u_i,u'_i]$ and $\pr_2(f_i(\D^3))=[v_i,v'_i]$~;
\item $e_i^*\F_0=f_i^*\F_0$ is the trivial foliation $dx_3\vert\D^3$~;
\item $\pr_2\circ f_i=\bar g_i\circ\pr_2\circ e_i$~.
\end{enumerate}

We perform on $M_0$ an elementary surgery of index $1$ by cutting
the interiors of $e_i(\D^3)$ and of $f_i(\D^3)$, and by pasting their boundary 2-spheres.
The points $e_i(x)$ and $f_i(x)$ are pasted, for every $x\in\partial\D^3$~.

We perform such a surgery on $M_0$ for every $1\le i\le p$~,
choosing of course the embeddings $e_i$~, $f_i$ two by two disjoint.
Let $M_1$ be the resulting manifold.

Obviously, $\F_0$ induces on $M_1$ a Morse foliation $\F_1$~, with
$2p$ singularities, one at every point $s_i:=e_i(0,0,-1)=f_i(0,0,-1)$~,
of Morse index $1$~;
and one at every point $s'_i:=e_i(0,0,+1)=f_i(0,0,+1)$~, of Morse
index $2$~. It is easy and standard to endow $M_1$ with a smooth structure,
such that $\F_1$ is of class $C^{\infty,r}$, and smooth in a neighborhood
of the singularities, minus the singularities.

By (4), the holonomy of $\F_1$ on the transversal $*\times T\cong T$ is
generated by the local transformations $g_i$~. That is, it coincides with $G$~.

Up to now, we have not used fully the fact that $G$ is compactly generated. Now, we
point a consequence of this fact, which is
actually its geometric translation.

Consider in general
 some Morse foliation $\X$ on some $3$-manifold $X$~, and some singularity
$s$ of index $1$~. On some neighborhood of $s$~, the Morse foliation $\X$
admits the first integral  $Q:=-x_0^2+x_1^2+x_2^2$ in some continuous local 
coordinates $x_0$~, $x_1$~, $x_2$, smooth except maybe at the singularity. The two components of the singular cone at $s$~,
namely $Q\mun(0)\cap\{x_0<0\}$ and $Q\mun(0)\cap
\{x_0>0\}$~, may either belong to the same leaf of the regular foliation
$\X\vert(X\setminus Sing(\X))$~, or not. If they do, then there is a loop
$\lambda:[0,1]\to X$ such that
\begin{itemize}
\item $\lambda(0)=\lambda(1)=s$~;
\item $\lambda$ is tangential to $\X$~;
\item $\lambda(t)\notin Sing(\X)$ for every $0<t<1$~;
\item $x_0(\lambda(t))\le 0$
(resp. $\ge 0$)
for every $t$ close enough to $0$ (resp. $1$).
\end{itemize}

Such a loop has a holonomy germ $h(\lambda)$ on the pseudo-transversal arc $x_0=x_1=0$~,
$x_2\ge 0$~. This is the germ at $0$ of some homeomorphism of the nonnegative
half-line.
\begin{dfn}\label{levitt_dfn} If moreover 
 the holonomy $h(\lambda)$ is the identity,
then we call $\lambda$ a \emph{Levitt loop}
for $\X$ at $s$~.
\end{dfn}

In the same way,
at every singularity of index $2$~, the Morse foliation $\X$ admits the first integral
$Q':=x_0^2-x_1^2-x_2^2$ in some local coordinates $x_0$~, $x_1$~, $x_2$~. The notion of a Levitt loop is defined
symmetrically by reversing
the transverse orientation of $\X$~.

\begin{lem}\label{levitt_lem}
The Morse foliation $\F_1$ admits a Levitt loop at every singularity.
\end{lem}

\begin{proof} Consider e.g. a singularity $s_i$ of index $1$~.
In $M_0$~, let $a$ be a path from the point $(*,u_i)$ to the point $e_i(0,0,-1)$ in the
sphere $\S^2\times u_i$~; and let $b$ be a path from $(*,v_i)$ to $f_i(0,0,-1)$ in the
sphere $\S^2\times v_i$~. Then, in $M_1$~,
 the path $ab\mun$ is tangential
to $\F_1$ and passes through $s_i$~.
Obviously, the holonomy $h(ab\mun)$ of $\F_1$
 on $*\times T\cong T$
along this path is well-defined on the right-hand side of $u_i$~.
That is, $h(ab\mun)$ is a germ of homeomorphism of $T$ from
some interval $[u_i,u_i+\epsilon)\subset T$ onto some interval
$[v_i,v_i+\eta)\subset T$~.
 By properties (2) through (4) above,$$h(ab\mun)=g_i\vert[u_i,u_i+\epsilon).$$

On the other hand, recall that $\bar g_i\in G$~.
Since $G$ is the holonomy pseudo-group of $\F_1$ on $*\times T$~, there is in $M\setminus
Sing(\F_1)$~, a path $c$ from $u_i$ to $v_i$~, tangential
to $\F_1$~, and whose holonomy on $*\times T$
is the germ of $\bar g_i$ at $u_i$~. Then, $\lambda:=a\mun cb$
is a Levitt loop at $s_i$~.
\end{proof}

To simplify the argument in the rest of the construction, it is convenient
 (although in fact not necessary) that $\F_1$ admit
at each singularity a \emph{simple}
Levitt loop.
We can get this extra property as follows. Let $s$ be a singularity of $\F_1$~, let
$\lambda$ be a Levitt loop for $\F_1$ at $s$~, and let
 $L$ be the leaf singular at $s$~, containing
$\lambda$~. After a generic perturbation of $\lambda$ in $L$~,
the loop $\lambda$ is immersed and self-transverse in $L$~.
Let $x$ be a self-intersection point of $\lambda$~. Since $\F_1$ is taut,
there passes through $x$
an embedded transverse circle $C\subset M_1$~, disjoint from $*\times T$~.
 We perform a surgery on $M_1$~,
cutting a small tubular neighborhood $N\cong\D^2\times\S^1$
 of $C$ in $M_1$~, in which $\F_1$ is the foliation by the $\D^2\times t$'s; and we glue
$\Sigma\times\S^1$~, where $\Sigma$ is the compact connected orientable surface of
genus $1$ bounded by one circle, foliated by the $\Sigma\times t$'s.
The holonomy pseudo-group of the foliation $\F_1$ on $*\times T$ is not changed.
After the surgery, there is at $s$ a Levitt loop with one less self-intersection. Of the
two pieces of $\lambda$ that passed through $x$~, now one passes in the new handle and is disjoint from the other.

  After a finite number of such surgeries,
for every $1\le i\le p$~, the Morse foliation $\F_1$
admits at the singularity $s_i$ (resp. $s'_i$) a simple Levitt loop $\lambda_i$
(resp. $\lambda'_i$).

Fix some $1\le i\le p$~. We shall somewhat cancel the pair of
 singularities $s_i$ and $s'_i$
of $\F_1$~,
at the price of a surgery on $M_1$~, without changing the
 transverse structure of $\F_1$~.

First, we use fully the fact that $G$ is taut: there is a
path $p_i:[0,1]\to M_1$ from $p_i(0)=s'_i$ to $p_i(1)=s_i$~, and
 positively transverse to $\F_1$ except at its endpoints.
 
The geometry is as follows (figure  \ref{one_fig}).
\begin{figure}
\includegraphics*[scale=0.45, angle=0]{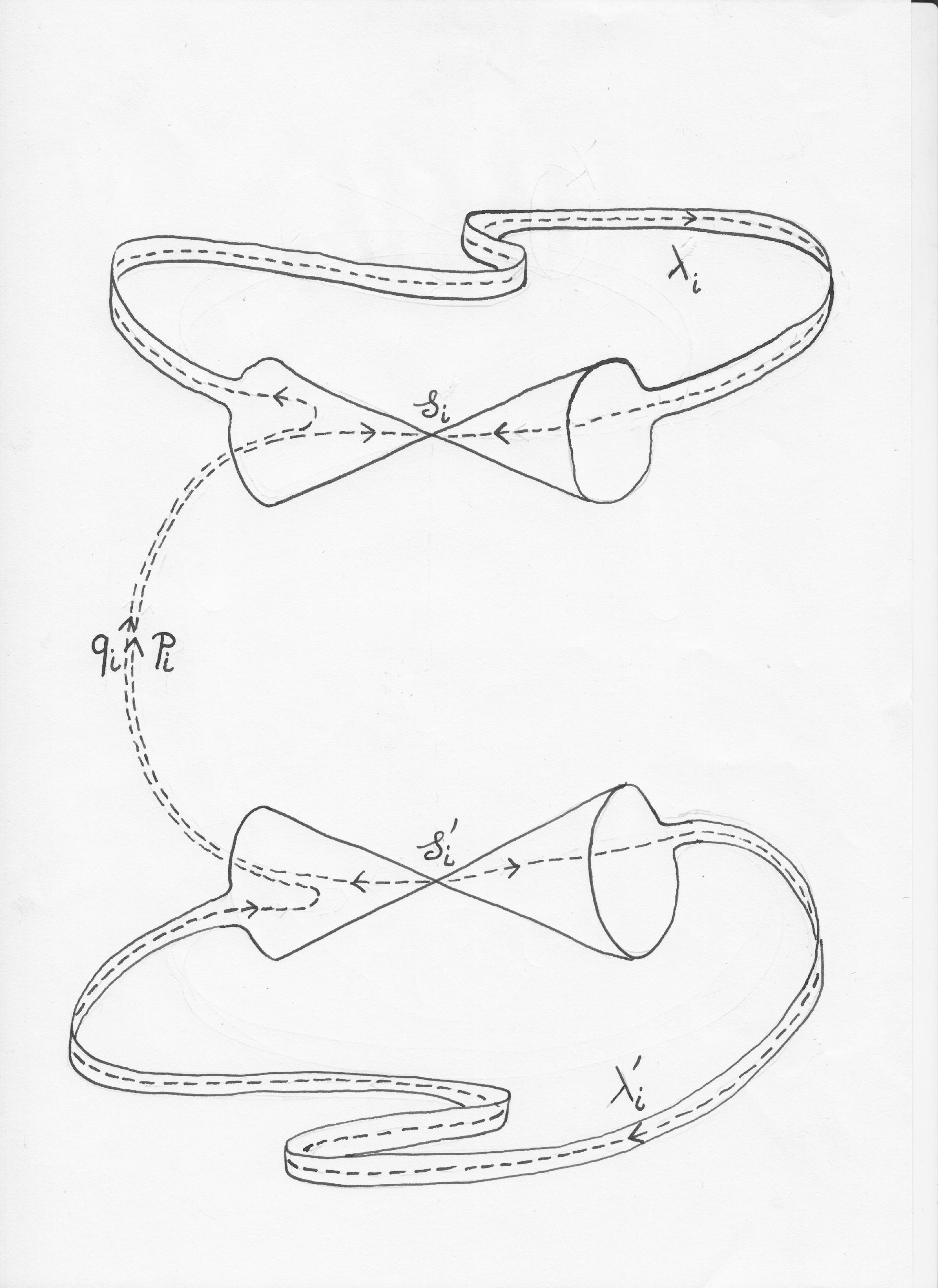}
\caption{}
\label{one_fig}
\end{figure}
 Let $Q(x_1,x_2,x_3)$ be a quadratic form
 of Morse index $1$ with respect to some local system
of coordinates at $s_i$~, which is a local first integral for $\F_1$~.
Then, $p_i$ arrives at $s_i$ by one of the two components of the cone $Q<0$~.
Reversing if necessary the orientation of $\lambda_i$, one can arrange that
$\lambda_i$ quits $s_i$ in the boundary of the same half cone.
  Symmetrically,
  let $Q'(x_1,x_2,x_3)$ be a quadratic form of Morse index $2$ with respect to some local system
of coordinates at $s'_i$~, which is a local first integral for $\F_1$~.
Then,
 $p_i$ quits $s'_i$ by one of the two components
 of the cone $Q'>0$~.
 Reversing if necessary the orientation of $\lambda'_i$, one can arrange that
$\lambda'_i$ arrives at $s'_i$ in the boundary of the same half cone.

We shall perform a surgery on $M_1$~, and modify $\F_1$~, in an arbitrarily small neighborhood of
$\lambda'_i\cup p_i\cup\lambda_i$~,
to cancel the singularities $s_i$~, $s'_i$~, without changing the holonomy pseudo-group
of the foliation.

To this aim, the composed path $\lambda'_ip_i\lambda_i$
(that is, $\lambda'_i$ followed by $p_i$ followed by $\lambda_i$) is homotoped, relatively to its endpoints,
 into some path
$q_i$ also positively transverse to $\F_1$~, except at its endpoints $q_i(0)=s'_i$ and
$q_i(1)=s_i$~. The homotopy consists in pushing the two
tangential Levitt loops to some nearby, positively transverse paths, and
in  rounding the two corners; it is $C^0$-small.

Notice that $p_i$ and
 $q_i$ arrive at $s_i$ by the two opposite components of the cone $Q<0$~.
  Symmetrically,
 $p_i$ and $q_i$ quit $s'_i$ by the two opposite components
 of the cone $Q'>0$~.
 
By construction, for a convenient choice of the parametrization $t\mapsto q_i(t)$,
the transverse path  $q_i$ is \emph{$\F_1$-equivalent} to $p_i$~, that is,
the diffeomorphism $p_i(t)\mapsto q_i(t)$ belongs to the holonomy
pseudo-group of $\F_1$ on the union of the two transversal open arcs $p_i\cup q_i\setminus
\{s_i,s'_i\}$~.

After a small, generic perturbation of $p_i$ and $q_i$
relative to their endpoints $s_i$, $s'_i$, we arrange that $p_i$ and $q_i$ are two embeddings of
the interval into $M_1$~; that they are disjoint,
but at their endpoints; and also that they are disjoint from $p_j$~, $q_j$ for every $j\neq i$~,
and also disjoint from the transversal $*\times T$~.

Recall (definition \ref{morse_dfn}) that $\F_1$ is smooth
in a neighborhood of $s_i$ and $s'_i$ (but maybe at $s_i$ and $s'_i$).
 After a $C^r$-small perturbation
of $\F_1$ in some small neighborhood of $p_i$ and $q_i$, relative
to some small neighborhoods of $s_i$ and $s'_i$, the foliation $\F_1$ is smooth in some
neighborhood of $p_i\cup q_i$ (but maybe at $s_i$ and $s'_i$).
\medbreak
Now, we shall perform on $M_1$ an elementary surgery of index $2$ along every embedded circle
$p_i\cup q_i$ ($1\le i\le p$)
 (figure \ref{two_fig}).
\begin{figure}
\includegraphics*[scale=0.45, angle=0]{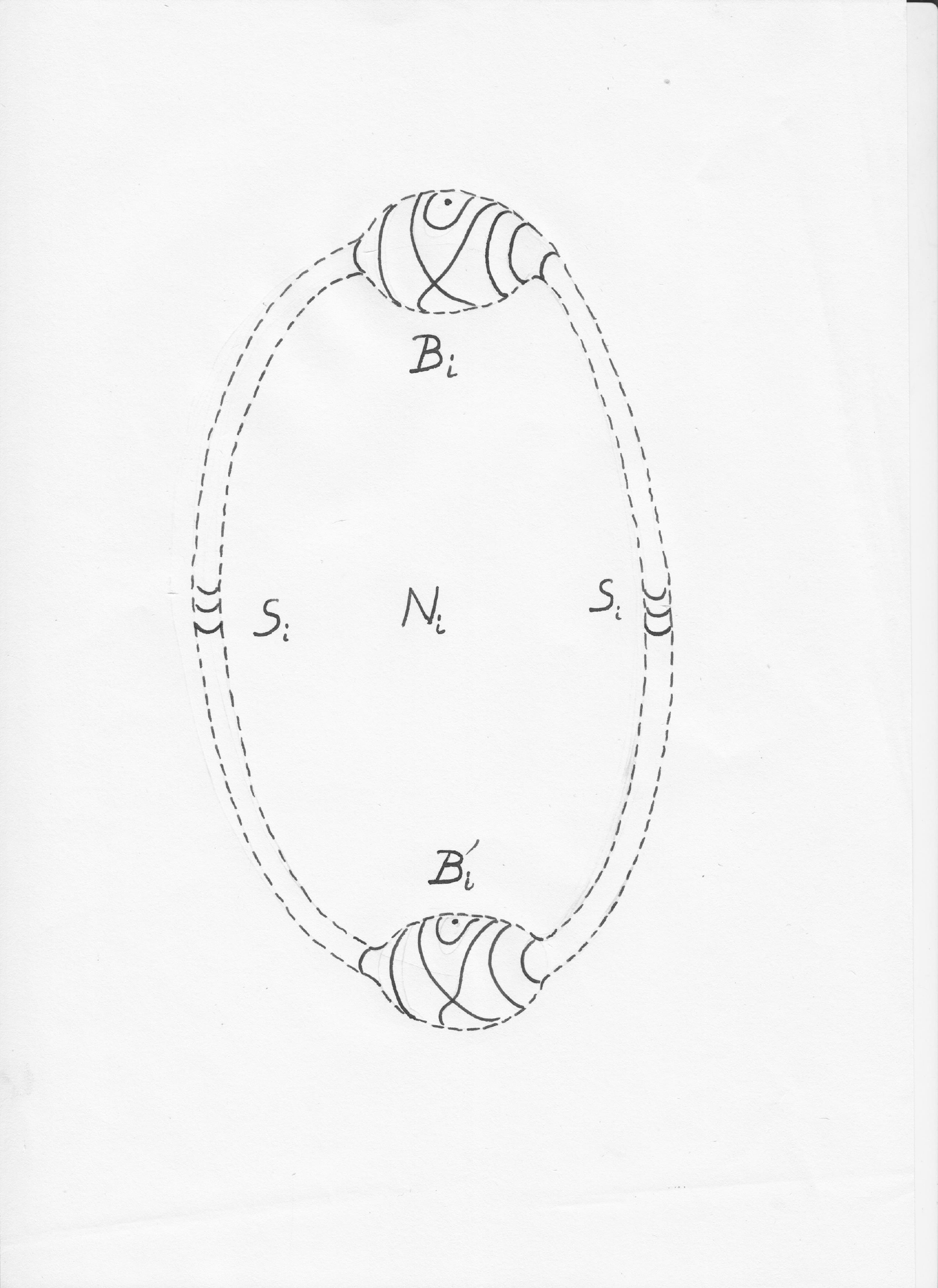}
\caption{}
\label{two_fig}
\end{figure}
  That is, we cut some small tubular neighborhood
$N_i\cong\S^1\times\D^2$ of $p_i\cup q_i$~, and we paste $\D^2\times\S^1$
 (here the choice of the framing is irrelevant). We shall
 obtain
  a closed $3$-manifold
$M$~.
We shall, for a convenient choice of the $N_i$'s~, extend the foliation $\F_1\vert(M_1\setminus\cup_i N_i)$ to $M$~,
as a (regular) foliation, still admitting $*\times T$
as an exhaustive transversal, and whose holonomy pseudo-group on $*\times T$
will still be $G$~.

To this end, first notice that,
by definition \ref{levitt_dfn}, and since $\lambda_i$ (resp. $\lambda'_i$)
is a simple loop, there is some small open neighborhood $U_i$
(resp. $U'_i$) of
$\lambda_i$ (resp. $\lambda'_i$) in $M_1$~, such that
the foliation $\F_1$ admits in $U_i\setminus s_i$ (resp. $U'_i\setminus s'_i$)
 a first integral $F_i$ (resp. $F'_i$)
 whose level sets are connected.
 Precisely, for every $t<F_i(s_i)$ (resp. $t>F'_i(s'_i)$), the level set $F_i\mun(t)$ (resp. $F'_i\mun(t)$) is
 an open disk. For every $t>F_i(s_i)$ (resp. $t<F'_i(s'_i)$), the level set $F_i\mun(t)$ (resp. $F'_i\mun(t)$) is
the connected orientable open surface of genus one with one end.

   Choose a compact $3$-ball $B_i\subset U_i$ containing $s_i$ and
    such that $F_i\vert B_i$ is
 topologically conjugate to a quadratic form $Q$ of signature $-++$~, with three different eigenvalues, on the unit ball. Choose a compact $3$-ball
  $B'_i\subset U'_i$ containing $s'_i$ and such that $F'_i\vert B'_i$ is
 topologically conjugate to a quadratic form $Q'$ of signature $--+$~, with three different eigenvalues,
  on the unit ball.
 Choose some tubular neighborhood $S_i$
  of the circle $p_i\cup q_i$~,
   so thin that $S_i\cap\partial B_i$ (resp.  $S_i\cap\partial B'_i$) is 
 contained in the cone $Q<0$ (resp. $Q'>0$)~, and such that $\F_1\vert S_i$
 is a foliation by disks, except on the intersections
 of $S_i$ with $Q\ge 0$ and with $Q'\le 0$~.
 Define $N_i:=B_i\cup B'_i\cup S_i$~.
 We can arrange that $N_i$ is a smooth solid torus.
 
 Then, after reparametrizing the values of $F_i$ and of $F'_i$~,
 they obviously extend to a function $F''_i$ on $N_i\setminus\{s_i,s'_i\}$ as follows.
 \begin{enumerate}
 \item $F''_i$ is a first integral for $\F_1$ on $N_i\setminus\{s_i,s'_i\}$~;
 \item $F''_i$ coincides with $F_i$ on $B_i$ and with $F'_i$ on $B'_i$~;
 \item  $F''_i\vert\partial N_i$ has exactly eight Morse critical points: two minima
 and two critical points
of index $1$ on $\partial B'_i$~, two critical points
of index $1$ and two maxima
 on $\partial B_i$~;
 \item The values of $F''_i$
at these critical points
are respectively $-2,-2,$ $-1,-1,$ $1,1,$ $2,2$~;
\item The sign of the tangency between $F''_i$
and $\partial N_i$ at each critical point is as follows: the descending gradient of $F''_i$ exits $N_i$
at the four critical points on $\partial B'_i$~, and enters $N_i$ at the four critical points
on $\partial B_i$~;
\item 
One has $F''_i(p_i(u))=F''_i(q_i(u))$ for every $u\in[0,1]$~.
\end{enumerate}

On the other hand, in the handle $H_i:=\D^2\times\S^1$~,
one has the function $h:=x_2(1+y_1^2)$~, where
$\D^2\subset\R^2$ (resp. $\S^1\subset\R^2$) is defined by $x_1^2+x_2^2\le 1$ (resp. $y_1^2+y_2^2=1$).
In $H_i$, the function $h$ has no critical point. On $\partial H_i$,
by (3), (4) and elementary Morse theory, $h\vert\partial H_i$ is smoothly conjugate to
$F''_i\vert\partial N_i$~. We attach $H_i$ to $M\setminus Int(N_i)$ so
that the functions $F''_i$ and $h$ coincide on $\partial N_i\cong\partial H_i$~.
 We extend $\F_1$ inside $H_i$ as the foliation defined by $h$~. 
  By (5), the sign of the tangency between $h$
and $\partial H_i$ at each singularity is the same as the sign of the tangency between $F''_i$ and $\partial N_i$~.
 So, the resulting foliation is regular.

Having done this for every pair of singularity $s_i$~, $s'_i$~, $i=1,\dots, p$~,
we get a regular foliation $\F$ on a closed $3$-manifold $M$~.

We claim that $\F$ admits $*\times T$ as an exhaustive transversal,
and has the same holonomy pseudo-group $G$ as $\F_1$ on $*\times T$~.
Obviously, $\F$ has no leaf contained in any $H_i$~. So, the claim
 amounts to verify the following. Let $\gamma:[0,1]\to M_1$ (resp. $M$) be a path tangential
to $\F_1$ (resp. $\F$) and whose endpoints belong to $M_1\setminus\cup_i Int(N_i)$~. Then, there is a path $\gamma':[0,1]\to M$ (resp. $M_1$)
tangential to $\F$ (resp. $\F_1$) with the same endpoints, and such that the holonomy of $\F_1$ (resp. $\F$) along $\gamma$ is the same as the holonomy of $\F$
(resp. $\F_1$) along $\gamma'$~.

We can assume that $\gamma$ is contained in some
 $N_i$ (resp. $H_i$), with endpoints
on $\partial N_i=\partial H_i$~. Let $t:=F''_i(\gamma(0))=h(\gamma(0))
=F''_i(\gamma(1))=h(\gamma(1))$~.

First, consider the case where $\gamma$ is contained in $N_i$
and tangential to $\F_1$~. There are three subcases, depending on $t$.

 First subcase: $F''_i(s'_i)<t<F''_i(s_i)$~. Then, the level set $F''_i\mun(t)$ is the disjoint union of two disks, so $\gamma$
has the same endpoints as some path $\gamma'$
contained in $\partial F''_i\mun(t)$~, and we are done.

Second subcase: $F''_i(s_i)\le t<2$~. Then,
consider the level set $F_i\mun(t)\subset U_i$~. Obviously, the intersection of
this level set with $U_i\setminus Int(B_i)$ is connected:
a pair of pants when $t<1$, a pair of pants when $t>1$, and it is also connected when $t=1$. So, $\gamma$
has the same endpoints as some path $\gamma'$
contained in this intersection, and we are done. (If the endpoints of $\gamma$
do not lie on the same connected component of the boundary of the annulus
$F_i\mun(t)\cap B_i$~, then the path $\gamma'$ will be close to the Levitt loop $\lambda_i$).

 The third and last subcase
$-2<t\le F''_i(s'_i)$ is symmetric to the second.

Now, consider the second case, where $\gamma$ is contained in $H_i$
and tangential to $\F$~.

In the subcases $-2<t<-1$ and $1<t<2$~,
the level set $h\mun(t)$ is the disjoint union of two disks. Thus, 
 $\gamma(0), \gamma(1)$ are also the endpoints of some path $\gamma'$
contained in $\partial(h\mun(t))$~, and we are done. The like holds for $t=-2, -1, 1$ or $2$.

In the subcase $-1<t<1$, the level set $h\mun(t)$ is an annulus. If 
 $\gamma(0), \gamma(1)$ belong to a same component of $\partial(h\mun(t))$~, we are done.
In the remaining sub-subcase, $\gamma(0), \gamma(1)$ belong to the two different circle components of $\partial(h\mun(t))$~.
By (6), these two circles are also the boundaries of the two disk leaves of $\F_1\vert S_i$ through $p_i(u)$ and $q_i(u)$~,
for some $u\in(0,1)$~. Now, recall that
the diffeomorphism $p_i(u)\mapsto q_i(u)$ between the transversals $p_i$ and $q_i$
belongs to the holonomy pseudo-group of $\F_1$ on $p_i\cup q_i$~. In other words, there is a path $\gamma':[0,1]\to M_1$
tangential to $\F_1$ with the same endpoints as $\gamma$~, and such that the holonomy of $\F$ along $\gamma$ is the same as the holonomy of
 $\F_1$ along $\gamma'$~.

Theorem A is proved in the case of a pseudo-group $(G,T)$ without boundary.
\medbreak
 Now, let us prove theorem A for a taut, compactly generated
  pseudo-group $(G,T)$
 such that $T$ has a boundary. One can assume that
 $(G,T)$ is connected. Thus, one is reduced to the case where $T$
 is a finite disjoint union of compact intervals (proposition \ref{connected_taut_pro}).
 
 The construction is much the same as in the case without boundary. We stress the few differences.
 
 We start from the manifold $M_0:=\S^2\times T$~.
For some of the generators $g_i$~, their domains and images meet the boundary,
i.e. they are semi-open intervals. Consider for example a $g_i$ whose domain
meets the positive boundary $\partial_+T$ (the boundary points where the
 tangent vectors which are positive with respect to the orientation of $T$ ,
exit from $T$). That is, $Dom(g_i)=(u_i,u'_i]$
and $Im(g_i)=(v_i,v'_i]$ and $Dom(g_i)\cap\partial T=u'_i$
and $Im(g_i)\cap\partial T=v'_i$~.

Such a generator will be introduced in the holonomy of the foliation by performing,
somewhat, a \emph{half} elementary surgery of index $1$ on the manifold $M_0$~. Put
for every $n$~:
$$\demi\D^n:=\{(x_1,\dots,x_n)\in\R^n\ \vert\ x_1^2+\dots+x_n^2\le 1,
 x_n\le 0\} $$

 Its boundary splits as the union of $\D^{n-1}$ (the subset defined in $\demi\D^n$ by $x_n=0$) and $\demi\S^{n-1}$
 (the subset defined in $\demi\D^n$ by $x_1^2+\dots+x_n^2=1$)~. Fix some extension $\bar g_i\in G$~.
 
Choose two embeddings $e_i:\demi\D^3\to M_0$ and $f_i:\demi\D^3\to M_0$
such that
\begin{enumerate}
\item $e_i\mun(\partial M_0)=f_i\mun(\partial M_0)
=\D^2$~;
\item $e_i(\demi\D^3)$ and $f_i(\demi\D^3)$ are disjoint from each other and from $T\times\ast$~;
\item $\pr_2(e_i(\demi\D^3))=[u_i,u'_i]$ and $\pr_2(f_i(\demi\D^3))=[v_i,v'_i]$~;
\item $e_i^*\F_0=f_i^*\F_0$ is the trivial foliation $dx_3$ on $\demi\D^3$~;
\item $\pr_2\circ f_i=\bar g_i\circ\pr_2\circ e_i$~.
\end{enumerate}

We perform on $M_0$ a surgery by cutting
 $e_i(\demi\D^3\setminus\demi\S^2)$ and $f_i(\demi\D^3\setminus\demi\S^2)$ and by pasting
 $e_i(\demi\S^2)$ with $f_i(\demi\S^2)$~.
The points $e_i(x)$ and $f_i(x)$ are pasted, for every $x\in\demi\S^2$~.

This surgery produces a single singularity
 $s_i:=e_i(0,0,-1)=f_i(0,0,-1)$~,
of Morse index $1$~.

The case of a generator $g_i$ whose domain meets $\partial_-T$ is of course symmetric.

 After performing a surgery for every generator, we get a resulting compact manifold
 $M_1$~, and a Morse foliation $\F_1$ induced on $M_1$ by $\F_0$~, with some singularities of indices $1$ and $2$~.
The boundary of $M_1$ is the disjoint union of two closed connected
surfaces $\partial_-M_1$~, $\partial_+M_1$~, both
 tangential to $\F_1$. At every point of
 $\partial_-M_1$ (resp. $\partial_+M_1$)~, the tangent vectors positively
transverse to $\F_1$ enter into (resp. exit from) $M_1$~.
The holonomy pseudo-group of $\F_1$ on $*\times T$ coincides with $G$~.

These singularities are eliminated one after the other (\emph{not} by pairs).
Let us eliminate e.g. a singularity $s_i$ of index $1$~.

 On the one hand, by tautness, there is
a path $p_i$~, positively transverse to $\F_1$ but at $s_i$~, from $p_i(0)\in \partial_-M_1$
to $p_i(1)=s_i$~.

On the other hand, by compact generation, $\F_1$
admits a Levitt loop $\lambda_i$
at $s_i$~.
We can arrange that $\lambda_i$ is a simple loop: if it has
a transverse self-intersection $x$~, then, by tautness, through $x$ there 
passes an arc $A$
embedded in $M_1$~, positively transverse to $\F_1$~, and whose endpoints lie on $\partial M_1$~. We
perform a surgery on $M_1$ along $A$~, cutting a small tubular neighborhood
$\cong\D^2\times[0,1]$ and pasting $\Sigma\times[0,1]$ (recall that $\Sigma$
is the disk endowed with a handle: see the paragraph below the proof of lemma \ref{levitt_lem}). The holonomy pseudo-group
of the foliation is not changed. After the surgery, $s_i$ admits a Levitt loop with
one less self-intersection.

The composed path $p_i\lambda_i$ is homotoped to
a path $q_i$ positively transverse to $\F_1$~,
arriving at $s_i$ through the component of the cone $Q<0$
opposite to that of $p_i$~; and $q_i$ is $\F_1$-equivalent to $p_i$~.
During the homotopy, the extremity endpoint $s_i$ is fixed, but the origin endpoint
moves in $\partial_-M_1$~. One arranges that $p_i\cap q_i=s_i$~.

The singularity $s_i$ is eliminated by, somewhat, a \emph{half} elementary surgery of index $2$
along the arc $p_i\cup q_i$ : one cuts a small tubular neighborhood of this arc,
$N_i\cong[0,1]\times\D^2$~, such that
 $N_i\cap\partial M_1\cong\{0,1\}\times\D^2$~; and one pastes $\demi\D^2\times\S^1$ foliated by the restricted function $h\vert(\demi\D^2\times\S^1)$~.
Every arc $(\demi\S^1)\times\theta\in\partial(\demi\D^2)\times\S^1$ is identified with $[0,1]\times\theta\in[0,1]\times\partial\D^2$~.
 The details are just like in the case without boundary.

\subsection{Proof of proposition \ref{reduction_pro}} We are now given
a compact manifold $M$ of dimension $n\ge 4$ endowed with a codimension $1$,
taut
foliation $\F$; and we have to find in $M$ a proper hypersurface
$M'$
transverse to $\F$, such that the inclusion induces a bijection between the spaces of leaves
 $M'/(\F\vert M')$ and  $M/\F$.
 
 To fix ideas, we consider only
the case where $M$ is closed connected and where $\F$ is smooth.

  Endow $M$ with an auxiliary Riemannian metric.
 Write $\rho(\F)$ the infimum of the injectivity radii of the leaves.
 
Fix a positive length $\delta<\rho(\F)/4$ so small that the following \emph{tracking}
property holds
for every leaf $L$ of $\F$ and  for every locally finite, $\delta$-dense
subset $A\subset L$ (in the sense that every point of $L$
is at distance \emph{less than $\delta$} from some point of $A$).
For every shortest geodesic segment $[a,b]$ whose
endpoints lie in $A$ and whose length is less than $\rho(\F)/2$,
there exists in $A$ a finite sequence $a_0=a$, \dots, $a_\ell=b$,
such that $d(a_{i-1},a_{i})<2\delta$ ($1\le i\le\ell$), and that the shortest geodesic segments
$[a_0,a_1]$, \dots, $[a_{\ell-1},a_\ell]$, $[b,a]$ form a simple loop
bounding a $2$-disk embedded in $L$.

Choose a circle $T$ embedded into
 $M$ transversely to $\F$, and such that $T\cap L$ is $\delta$-dense in every leaf $L$.
 
 Write $G$ the holonomy pseudo-group of $\F$
 on $T$. For any $g\in G$ and $r>0$, say that $g$ is \emph{$r$-short}
if for every $t\in Dom(g)$, the distance from $t$ to $g(t)$ in the leaf of
$\F$ through $t$ is less than $r$. At every point of $T$, one has
only finitely many $2\delta$-short germs of local transformations of $T$ belonging
to $G$. Thus, one has
 a finite family $g_1,\dots, g_p\in G$ such that
 every domain $Dom(g_i)$ is an interval $(t_i,t'_i)\subset T$~; and that every $2\delta$-short germ in $G$
 is the germ of some $g_i$ at some point of its domain. Moreover, one can arrange
 that  $g_1$, \dots, $g_p$ are $(\rho(\F)/2)$-short and
 $G$-extendable; and that
 the leaves through the endpoints $t_i, t'_i$
 are two by two distinct. One writes $\hat g_i$ the extension of $g_i$ to the compact interval $[t_i,t'_i]$~.

 The family $g_1$, \dots, $g_p$ generates $G$. Indeed, since $\delta<\rho(\F)/2$, for every leaf $L$, the fundamental groupoid of the pair
 $(L,L\cap T)$ is generated by the geodesic segments of length less than $2\delta$
 whose enpoints lie in $L\cap T$. 
 
 One can arrange moreover that to each $\hat g_i$ is associated a fence (recall
 definition \ref{holonomy_dfn}) $f_i$, such that the
  image rectangles $Im(f_1)=f_1([0,1]\times[t_1,t'_1])$,
   \dots, $Im(f_p)=f_p([0,1]\times[t_p,t'_p])$ are two by two disjoint in $M$. Indeed,
   one first has the fences composed by the tangential shortest geodesic segments
 $[t,\hat g_i(t)]$ ($t\in[t_i,t'_i]$). Since the leaves are of dimension $n-1\ge 3$,
 after a fine enough
  subdivision of the domains of the $g_i$'s into smaller subintervals, and after
 a small generic perturbation of the arcs $[t,g_i(t)]$
  relative to their endpoints, the image rectangles are two by two disjoint. Let
  $s_i:=f_p(1/2,t_i)$ (resp. $s'_i:=f_p(1/2,t'_i)$) be the middle of the lower (resp. upper)
    edge of each fence.

   \medbreak
    The rest of the proof of proposition \ref{reduction_pro} is alike the proof of theorem A, except that
 the construction is made inside $(M,\F)$. The dimension of the construction
    here is $n-1$, rather than $3$ as it was in theorem A, but this does not make
    any substantial difference. Here is a sketch.
    
     Let $K:=T\cup Im(f_1)\cup\dots\cup Im(f_p)$, a $2$-complex embedded into $M$.
 One has in $M$ a small compact neighborhood $\Omega$ of
     $K\setminus\{s_1,\dots,s_p,s'_1,\dots,s'_p\}$
    whose smooth boundary $M_1:=\partial\Omega$ is much like in the proof of theorem $A$.
 Precisely, $s_1,\dots,s_p,s'_1,\dots,s'_p\in M_1$~; and
    $M_1$ is transverse to $\F$ but at each $s_i$ (resp. $s'_i$), where $\F_1:=\F\vert M_1$ has a Morse singularity of index
    $1$ (resp. $n-2$). One has in $M_1$ a circle $*\times T$ close and parallel to $T$,
    transverse to $\F_1$, and meeting
    every leaf of $\F_1$, such that the holonomy of $\F_1$ on
    $\ast\times T\cong T$ is generated by $g_1$, \dots, $g_p$. That is, it coincides
    with $G$.
    
    Now, we use the tracking property to find some convenient Levitt loops. Consider any $s_i$~. In the leaf $L_i$ of $\F$ through $s_i$,
the shortest geodesic segment $[t_i,\hat g_i(t_i)]$ has length less than $\rho(\F)/2$, thus it is tracked by a piecewise geodesic path 
 $a_0=t_i$, \dots, $a_\ell=\hat g_i(t_i)$,
such that $d(a_{i-1},a_{i})<2\delta$ ($1\le i\le\ell$); and
$[a_0,a_1]$, \dots, $[a_{\ell-1},a_\ell]$, $[\hat g_i(t_i),t_i]$ form a simple loop $\lambda_{\rm geod}$
bounding a $2$-disk embedded in $L_i$. Close to $\lambda_{\rm geod}$~, one has a loop $\lambda_K$ in $L_i\cap K$~.
Close to $\lambda_K$~, one has a loop $\lambda_i$ in $L_i\cap M_1$~, passing through $s_i$~. Obviously, $\lambda_i$ is
a Levitt loop for $\F_1$ at $s_i$~.
 If the fences $f_1$~,\dots, $f_n$ have been
 taken close enough to the geodesic ones, and if the neighborhood $\Omega$ of $K$ has been taken thin enough,
then $\lambda_i$ is also simple, and  bounds also a disk $\Delta_i$ embedded in $L_i$~.

 The like holds at every $s'_i$~, and yields a simple Levitt loop
 $\lambda'_i$ bounding a disk $\Delta'_i$ embedded in the leaf of $\F$ through $s'_i$~.
 The leaves of $\F$ through the singularities of $\F_1$ being two by two distinct, the disks are two by two disjoint.

	Like in the proof of theorem $A$, we have in $M_1$ a simple path $p_i$
	from $s'_i$ to $s_i$, positively
	transverse to $\F_1$ but at its endpoints. The composed path $\lambda'_ip_i
	\lambda_i$ is perturbated in $M_1$, relative to his endpoints,
	into some simple path $q_i$ transverse to $\F_1$ and disjoint from $p_i$,
	 but at its endpoints. The union of the
	disks $\Delta_i$ and $\Delta'_i$ with a thin strip is perturbated into
	 a $2$-disk $\Delta''_i$ (rather obvious on figure \ref{one_fig}) such that
	 \begin{itemize}
	 \item $\Delta''_i$ is embedded into $M$;
	  \item $\partial\Delta''_i=\Delta''_i\cap M_1=p_i\cup q_i$;
	  \item $\Delta''_i$ is transverse to $\F$ and $\F\vert\Delta''_i$
	  is the foliation of the $2$-disk by parallel straight segments.
	  \end{itemize}
	  
	  The hypersurface $M'\subset M$ is built from $M_1$
	   by cutting, for every $1\le i\le p$,
	  a small tubular neighborhood of $p_i\cup q_i$ in $M_1$, diffeomorphic
	  with $\S^1\times\D^{n-2}$,
	  and pasting the boundary of the sphere bundle normal to $\Delta''_i$
	  in $M$, diffeomorphic to $\D^2\times\S^{n-3}$.

 \section{proof of theorems B and C}\label{BC_section}

 \subsection{Examples: realizing the homothety pseudo-groups}\label{homothety_sbs}

 First, we discuss the realization of some elementary but fundamental examples:
 the homothety pseudo-groups. They constitute the most simple nontaut,
 compactly generated pseudo-groups.
 
Given some positive real numbers $\lambda_1$~, \dots, $\lambda_r$~,
 let $G(\lambda_1,\dots,\lambda_r)$ be the pseudo-group
of local transformations of the real line
generated by the
 homotheties $t\mapsto\lambda_1 t$~, \dots, $t\mapsto\lambda_r t$~.
 We assume that
 the family $\log\lambda_1$~, \dots, $\log\lambda_r$
 is of linear rank $r$ over $\Q$~.
 
 For $r=1$~, the pseudo-group
  $G(\lambda_1)$ has two obvious realizations of interest.
 The first is
  on the annulus $A:=\S^1\times[0,1]$~.
   The compact leaf is
    $\S^1\times (1/2)$;
  the other leaves are transverse to $\partial A$ and spiral towards $\S^1\times (1/2)$.
  The second realization is on $\partial(A\times\D^2)\cong\S^2\times\S^1$~.
   The compact leaf is a $2$-torus, and splits $\S^2\times\S^1$ into
  two Reeb components.
  
   On the contrary, $G(\lambda_1)$ is \emph{not} realizable on $T^2$. Indeed,
   the foliation would be transversely oriented and have a single compact leaf,
   whose linear holonomy would be nontrivial, a contradiction.

 The case $r=2$ is analogous. The torus $T^2$ is endowed with the angle coordinates $x, y$~. One realizes
 $G(\lambda_1,\lambda_2)$ on $V:=T^2\times[0,1]$ by a foliation $\F(\lambda_1,\lambda_2)$
  transverse to both boundary tori,
 where its trace is the linear irrational foliation $dx\log\lambda_1+dy\log\lambda_2=0$~. The torus  $T^2\times(1/2)$ is a compact leaf;
the other leaves spiral towards it.
 
Notice that $G(\lambda_1,\lambda_2)$ is \emph{not} realizable by any foliation $\F$
on any closed, orientable $3$-manifold $M$~. For, by contradiction,
 $\F$
would have a unique compact leaf $L$ diffeomorphic to $T^2$~, along which $M$
 would split into two compact $3$-manifolds $M'$~, $M''$~.
On $\R\setminus 0$~,
the differential $1$-form $dt/t$ is invariant by $G(\lambda_1,\lambda_2)$~.
There would correspond on $M\setminus L$ a nonsingular closed $1$-form $\omega$ of rank $r=2$~, such that $\F\vert(M\setminus L)=\ker\omega$~.
 In $H^1(M';\R)$~,
 the de Rham cohomology class $[\omega]$
 decomposes as $(\log\lambda_1)e_1+(\log\lambda_2)e_2$~, with $e_1,e_2\in H^1(
 M';\Z)$~.
 The restriction $[\omega]\vert L\in H^1(L;\R)$ is of rank $2$~,
 being the class of the linear holonomy of $\F$ along $L$~. Thus, $e_1\vert L$ and  $e_2\vert L$ are not $\Q$-colinear
  in $H^1(L;\Q)$~.
 Since $L$ is a $2$-torus, $(e_1\vert L)\wedge(e_2\vert L)\neq 0$ in
 $H^2(L;\Z)$~. In other words, $e_1\wedge e_2\in H^2(M';\Z)$ is nonnull
 on the fundamental class of $\partial M'$~. This contradicts Stokes theorem,
  $M'$ being an orientable compact $3$-manifold.

One can ask if things would turn better if one dropped the condition that the realization be a \emph{tangentially orientable}
foliation. It is not difficult to see that the answer is negative: $G(\lambda_1,\lambda_2)$ is also {not} realizable by any foliation $\F$~ ,
even not orientable,
on any closed $3$-manifold $M$~. This is left as an exercise. I thank the referee for pointing out a mistake at this point
 in the first version of this paper.

      For \emph{every} $r\ge 2$~, the pseudo-group $G(\lambda_1,\dots,\lambda_r)$ is 
  realizable on a closed orientable $4$-manifold. Indeed, in a first place,
   for each $2\le i\le r$~,
 just as above,
 realize $G(\lambda_1,\lambda_i)$ by a foliation $\F(\lambda_1,\lambda_i)$ on $V:=T^2\times[0,1]$~.
  So,  $G(\lambda_1,\lambda_i)$ is also realized by the pullback $\F_i$
  of $\F(\lambda_1,\lambda_i)$ in the $4$-manifold
  $$M_i:=\partial(V\times\D^2)
  \cong T^2\times\S^2$$
  The compact leaf $L_i$ of $\F_i$ is the $3$-torus $T^2\times\S^1$.
 For each $i=3,\dots,r$,  in $L_2$ and in $L_i$, we pick some embedded circle $C_i\subset L_2$ (resp. $C'_i\subset L_i$)
  parallel to the first circle factor: the holonomy of $\F_2$ (resp. $\F_i$) along $C_i$
  (resp. $C'_i$) is the germ of $t\mapsto\lambda_1t$
  at $0$~. We arrange that $C_3$, \dots, $C_r$ are two by two disjoint. The loop $C_i$ (resp. $C'_i$) has in $M_2$ (resp. $M_i$) a small tubular neighborhood $N_i$ (resp. $N'_i$) $\cong\D^3\times\S^1$~, on
  the boundary of which $\F_2$ (resp. $\F_i$) traces a foliation composed of two Reeb components,
  realizing $G(\lambda_1)$~.
  We cut from $M_2$~, \dots, $M_r$ the interiors of $N_3$~, \dots, $N_r$~,
  $N'_3$~, \dots, $N'_r$~. We paste every $\partial N_i$ with $\partial N'_i$~,
  such that $\F_2\vert\partial N_i$ matches $\F_i\vert\partial N'_i$~. We get
  a closed $4$-manifold with a foliation realizing $G(\lambda_1,
  \dots,\lambda_r)$~.
  \medbreak
   
The realization of pseudo-groups of homotheties \emph{with boundary} is much alike:
 let $2\mun G(\lambda_1,\dots,\lambda_r)$ be the pseudo-group
of local transformations of the \emph{half}-line $\R_{\ge 0}$
generated by some family of
 homotheties $t\mapsto\lambda_1 t$~, \dots, $t\mapsto\lambda_r t$~,
  of rank $r$~. Each of the above realizations of $G(\lambda_1,\dots,\lambda_r)$
  splits along its unique compact leaf into two realizations of
   $2\mun G(\lambda_1,\dots,\lambda_r)$.

  \subsection{Novikov decomposition for pseudo-groups, and hinges}
  
  Let $(G,T)$ be a compactly generated pseudo-group of dimension $1$~.
  
We consider the closed orbits (the orbits topologically
closed in $T$~).

\begin{lem}\label{closed_orbits_lem}
   The union of the closed orbits is topologically closed in $T$~.
   \end{lem}
   
 \begin{proof} We know no proof for this fact 
  in the pseudo-group frame.
  To prove it, we realize the pseudo-group, as in section \ref{A_section}, by
   a Morse foliation $\F$
  on a compact manifold $M$~. Since the homology of $M\setminus Sing(\F)$
  is of finite rank, Haefliger's argument \cite{hae1962} applies and shows that the union of the closed leaves is closed.
  \end{proof}
  
 We call a closed orbit \emph{isolated} (resp. \emph{left isolated})
 (resp. \emph{right isolated}) if it admits some neighborhood
 (resp. left neighborhood) (resp. right neighborhood) in $T$
 which meets no other closed orbit.

 By a \emph{component} of $(G,T)$ , one means
 a submanifold $T'\subset T$ of dimension $1$, topologically closed in $T$,
 and saturated for $G$~.
 
 By an \emph{$I$-bundle} (resp. an \emph{$\S^1$-bundle})
 we mean the pseudo-group generated by
 a finite number $r$ of \emph{global} diffeomorphisms on the compact interval
 (resp. the circle). It is of course realized on some compact $3$-manifold
 fibred over some closed surface (``suspension''). Every pseudo-group
 Haefliger-equivalent to an $I$-bundle (resp. an {$\S^1$-bundle})
 is also called an $I$-bundle (resp. an {$\S^1$-bundle}).
 The smallest possible $r$ is the \emph{rank} of the $I$-bundle
 (resp. $\S^1$-bundle).
 
 Any closed orbit whose isotropy group has infinitely many fix points
 bounds an $I$-bundle. Precisely,
 
 \begin{lem}\label{accumulation_lem}
 Let $G(t)\subset T$ be a closed orbit,
 and let $h_1$, \dots, $h_r$ be elements of $G$ whose germs at $t$ generate the isotropy
 group $G_t$. Assume that $h_1$, \dots, $h_r$ admit a sequence $(t_n)$ of
  common fix points other than $t$, decreasing (resp. increasing) to $t$.
  Put $I_n:=[t,t_n]$ (resp. $[t_n,t]$).
  
 Then, for every $n$ large enough,
 
  \begin{itemize}
  \item The restricted pseudo-group  $(G\vert I_n,I_n)$ is generated by
 $h_1\vert I_n$, \dots, $h_r\vert I_n$;
 
 \item  The $G$-saturation of $I_n$
 is an $I$-bundle component of $(G,T)$, Haefliger-equivalent to  $(G\vert I_n,I_n)$.
  
 \end{itemize}
 \end{lem}
 \begin{proof}

  One first reduces oneself to the case where $G(t)=\{t\}$, as follows.
Let $U:=(T\setminus G(t))\cup{t}$. By Baire's theorem,
every closed orbit is discrete. So, $U$ is open in $T$ and meets every orbit.
We change $(G,T)$ for $(G\vert U,U)$, which is also compactly generated
by proposition \ref{invariance_pro}.

So, we assume that  $G(t)=\{t\}$.
 
 Since $G$ is compactly generated,
 one has a topologically open, relatively compact $T'\subset T$ meeting
 every $G$-orbit (in particular $t\in T'$) such that $G\vert T'$
 admits a system of generators $g_1$~, \dots, $g_p$ which are  $G$-extendable.
 Let $\bar g_1$~, \dots, $\bar g_p\in G$ be some extensions.
 
If
  $t$ lies in the topological boundary of $Dom(g_i)$ with respect to $T$~,
   then we can avoid
 this by changing $g_i$ for $\bar g_i\vert(Dom(g_i)\cup(t-\epsilon,t+\epsilon'))$~, where $(t-\epsilon,t+\epsilon')$
 is relatively compact in $Dom(\bar g_i)\cap T'$~.
 The like holds for $Im(g_i)$~.
 Thus, after permuting the generators, for some $0\le q\le p$~, the point
 $t
 $ belongs to the domains and to the images of $g_1$~, \dots, $g_q$~;
 but $t$
 does not belong, nor is adherent, to the domains nor to the images of
 $g_{q+1}$~, \dots, $g_p$~.
 
  Also, restricting the domain of each $h_i$, we arrange
  that $h_i\in G\vert T'$ and that $h_i$ is $G$-extendable.
  Then, we can add the family $(h_i)$ to the family
  of generators $(g_i)$. So, we can assume
  that $r\le q$ and that $g_1=h_1$, \dots, $g_r=h_r$.
  
   Then, for every $r+1\le i\le q$~, the generator $g_i$ coincides with some
   composite of $g_1$, \dots, $g_r$ on some small compact neighborhood $N_i$ of $t$.
   We can change $g_i$ for $g_i\vert(Dom(g_i)\setminus N_i)$.
   Finally, we have obtained a family of generators $g_1$, \dots, $g_p$ for $G\vert T'$,
   such that
 $t$ belongs to the domains and to the images of $g_1=h_1$~, \dots, $g_r=h_r$~;
 but that $t$
 does not belong, nor is adherent, to the domains nor to the images of
 $g_{r+1}$~, \dots, $g_p$~. 
 
 For every $n$ large enough,
 $I_n$ is contained in $T'$ and in the domains of $g_1$, \dots, $g_r$;
  and $I_n$ is
  invariant by $g_1=h_1$, \dots, $g_r=h_r$; and $I_n$
 is disjoint from the supports of $g_{r+1}$, \dots, $g_p$. Thus, $I_n$
 is saturated for $G$, and $G\vert I_n$ is generated by $g_1\vert I_n$, \dots,
 $g_r\vert I_n$. The interval $I_n$ is an $I$-bundle component of $G$.
 \end{proof}
 
 We call an orbit \emph{essential} (with respect to $(G,T)$)
  if it meets no transverse positive loop
 and no transverse positive chain whose both endpoints lie on $\partial T$ . Every essential
 orbit is closed (obviously). In the union of the closed orbits,
  the union of the essential orbits is topologically closed (obviously) and open
  (by lemma \ref{accumulation_lem}).
  
    We call an $I$-bundle component
  of $G$ \emph{essential} (with respect to $(G,T)$) if its boundary orbits are essential
  with respect to $(G,T)$. Then, every
  closed orbit interior to this $I$-bundle is also essential with respect to $(G,T)$.

 The ``Novikov decomposition'' is well-known for foliations on compact manifolds.
 Every compact connected manifold endowed with a foliation of codimension one,
 either is an $\S^1$-bundle, or splits along finitely many compact leaves bounding some
 dead end components, into compact  components, such that each component is
 an $I$-bundle, or its interior is topologically taut. For compactly generated
 pseudo-groups, one has an analogous
 decomposition (exercise):
  
 \begin{pro}\label{novikov_pro} (Novikov decomposition)
  Let $(G,T)$ be a connected, compactly generated
   pseudo-group of dimension $1$. Assume that $(G,T)$ is not an $\S^1$-bundle.
  
  Then, $T$ splits, along finitely many essential orbits,
 into finitely many components
 $T_i$,
  such that  for each $i$:

 (a) the component  $(G\vert T_i,T_i)$ is an essential $I$-bundle,

  or
  
  (b)
 the interior of the component,
$(G\vert Int(T_i), Int(T_i))$, is taut.

 \end{pro}
     Novikov decompositions are functorial with respect to Haefliger
   equivalences: given a Haefliger
  equivalence between two pseudo-groups, to every Novikov decomposition of the
  one corresponds naturally a Novikov decomposition of the other.

We shall not use this decomposition under this form, nor prove it in general.
We need, under the hypotheses of theorems B and C, the more precise form
\ref{splitting_pro} below.

 \medbreak
 From now on, we assume moreover that in the
 compactly generated, $1$-dimensional pseudo-group $(G,T)$, every essential orbit is \emph{commutative,}
 that is, its isotropy group is commutative. By the \emph{essential rank} of $(G,T)$,
 we mean the supremum of the ranks of the isotropy groups of the essential orbits.

  The proof of theorems B and C somewhat consists in realizing independently every
  component of some Novikov decomposition, and pasting these realizations together.
  The interior of every component falling to (b) in proposition \ref{novikov_pro}, is realized on a closed
  $3$-manifold, thanks to  theorem A and to the following

 \begin{lem}\label{complement_lem}
 Let $(G,T)$ be a compactly generated pseudo-group of dimension $1$~.
 Let $G(t_0)\subset T$ be an isolated closed orbit, whose isotropy group is commutative.

 Then, the subpseudo-group $G\vert(T\setminus G(t_0))$ is also compactly generated.
 \end{lem}
 \begin{proof}
 
  We treat the case where the orbit $G(t_0)$ is contained in $\partial_-T$~. Of course,
  the case where it is contained in $\partial_+T$ is symmetric; and
   the case where
   it is contained in $Int(T)$
 is much alike.
 
  Since $G$ is compactly generated,
 one has a topologically open, relatively compact $T'\subset T$ meeting
 every $G$-orbit such that $G\vert T'$
 admits a system of generators $g_1$~, \dots, $g_p$ which are  $G$-extendable.
 
 Just as in the proof of lemma \ref{accumulation_lem},
 one can arrange that $G(t_0)=\{t_0\}$~ (in particular, $t_0\in T'$);
and that, for some $0\le r\le p$~, the point
 $t_0$ belongs to the domains and to the images of $g_1$~, \dots, $g_r$~;
 and that $t_0$
 does not belong, nor is adherent, to the domains nor to the images of
 $g_{r+1}$~, \dots, $g_p$~.
 
   Since $t_0$ is isolated as a closed orbit
 of $G$~, one has $r\ge 1$~.
  We can arrange moreover, to simplify notations, that the family $(g_i)$
 is symmetric: the inverse of every $g_i$ is some $g_j$~.

 The isotropy group of $t_0$ being commutative,
 there is a $u_0>t_0$ so close to $t_0$ that
 \begin{enumerate}
 \item  For every $r+1\le i\le p$~, the interval $[t_0,u_0]$ does not meet $Dom(g_i)$~;
 \item For every $1\le i,j\le r$ and every $t\in[t_0,u_0]$~,
 one has $t\in Dom(g_i)$ and $g_i(t)\in Dom(g_j)$ and $g_ig_j(t)=g_jg_i(t)$~.
 \end{enumerate}
 
Put  $T'':=T'\setminus[t_0,u_0]$ and $G_0:=G\vert(T\setminus t_0)$~.
 We shall show that
  every orbit of $G_0$ meets $T''$~, and that the pseudo-group
 $G_0\vert T''$ is generated by $g_1\vert T''$~, \dots, $g_p\vert T''$~.
  Every $g_i\vert T''$ being $G_0$-extendable, it will follow that $G_0$ is
 compactly generated.
 
To this end, define by induction two sequences $u_n\in[t_0,u_0]$ and $1\le i(n)\le r$~,
such that $u_{n+1}:=g_{i(n)}(u_{n})$ is the minimum of $g_1(u_{n})$~,
\dots, $g_r(u_n)$~.
 Because $t_0$ is isolated as a closed orbit
 of $G$~, there is no common fixed point for $g_1$~, \dots, $g_r$ in the interval
 $(t_0,u_0]$~. Thus, $(u_n)$ decreases to $t_0$~. Also,  for every $n\ge 0$~: $$g_{i(n)}\mun((u_{n+1},u_0])\subset(u_n,u_0]\cup T''\ \ \ (*)$$
In particular,  every orbit of $G_0$ meets $T''$~.
  
 Consider the germ $[g]_t$ of some $g\in G_0$ at some point $t\in Dom(g)$
  such that $t\in T''$ and $g(t)\in T''$~. Since the $g_i$'s generate $G\vert T'$~,
 this germ can be decomposed as a word $w$ in the germs of the generators: $$[g]_t=[g_{j(\ell)}]_{t(\ell-1)}\dots[g_{j(1)}]_{t(0)}$$
 where $1\le j(1), \dots, j(\ell)\le p$~, where $t(0)=t$~,
  and where for every $0\le k\le\ell$
 one has $t(k):=g_{j(k)}\circ\dots\circ g_{j(1)}(t)\in T'$~.
 
 We call the finite sequence $t(0)$~, \dots, $t(\ell)$ the \emph{trace} of $w$~.
 We have to prove that $[g]_t$ admits also a second such decomposition, whose
 trace is moreover contained in $T''$~.

  We make a double induction: on the smallest integer $n\ge 0$ such that the trace
  of $w$
  is disjoint from $[t_0,u_{n}]$~, and, if $n\ge 1$~, on the number of $k$'s
  for which $t_k\in(u_{n},u_{n-1}]$~.
  
  Assume that $n\ge 1$~. Let $1\le k\le\ell-1 $ be an index for which $t_{k}\in(u_n,u_{n-1}]$~.
  Consider the word
  $$w':=g_{j(\ell)}
  \dots g_{j(k+2)}
  g_{i(n-1)}
  g_{j(k+1)}
   g_{j(k)}
  g_{i(n-1)}\mun
  g_{j(k-1)}
  \dots g_{j(1)}
  $$
  By the property (2) above applied at the point $t_k$ to the pair
   $g_{i(n-1)}\mun, g_{j(k+1)}$ and to the pair  $g_{i(n-1)}\mun, g_{j(k)}\mun$~,
   the composite $w'$ is defined at $t$~, and $w'$ has the same germ at $t$ as $w$~.
   
  The trace of $w'$ at $t$ is the same as the trace of $w$~,
  except that $t(k)$ has been changed for the three points
  ${g_{i(n-1)}\mun(t(k-1))}$~,  ${g_{i(n-1)}\mun(t(k))}$ and  ${g_{i(n-1)}\mun(t(k+1))}$~. By ($\ast$),
  none of the three lies in $[t_0,u_{n-1}]$~. The induction is complete.
 \end{proof}

         The pasting of the realizations of the Novikov components
          will be a little
  delicate. The following
 notion allows us to take in account,
  with every commutative closed orbit, its isotropy group;
 and with every commutative $I$-bundle,
  the holonomy of its boundary orbits on the exterior side.
 
\begin{dfn}\label{thick_dfn}
 We call a pseudo-group $(\Gamma,\Omega)$ of dimension $1$
  a \emph{hinge} if $\Omega$ is an interval, either open, or compact, or semi-open;
  and
   if there exist a $\Gamma$-invariant compact
 interval $[a,b]\subset\Omega$, with $a\le b$, and a system of generators
  $\gamma_1$~, \dots, $\gamma_r$ for $\Gamma$, such that
 \begin{enumerate}
\item The domains and the images of $\gamma_1$~,\dots, $\gamma_r$
are intervals containing $[a,b]$~;
\item
For every $\gamma,\eta\in\Gamma$~, one has $\gamma\eta=\eta\gamma$ and
 $\gamma\mun\eta=\eta\gamma\mun$ and  $\gamma\mun\eta\mun=\eta\mun\gamma\mun$ wherever both composites are defined;
 \item Every neighborhood of $[a,b]$ in $\Omega$ meets every orbit of $\Gamma$~.
 \end{enumerate}
 \end{dfn}
We call $[a,b]$ the \emph{core}.
 Write $\partial\Omega$ the boundary
of $\Omega$ as a manifold, that is,
 the boundary points of $\Omega$ belonging to $\Omega$, if any.
By (3), every such boundary point copincides with $a$ or $b$~.
 The hinge is
\emph{degenerate} if $a=b$, in which case $a=b$ is the (unique) closed
 orbit of $(\Gamma,
\Omega)$.  The hinge is
\emph{nondegenerate} if $a<b$, in which case $[a,b]$ is the (maximal) $I$-bundle component of $(\Gamma,
\Omega)$. 
 The smallest possible $r$ is the \emph{rank} of $(\Gamma,\Omega)$~.

\begin{lem}\label{no_fix_point_lem}
Let $(\Gamma,\Omega)$ be a hinge. Then, there exists a local transformation in $\Gamma$
whose domain contains the core, and which is fix-point free outside the core.
\end{lem}

\begin{proof}
i) In case $\Omega$ coincides with the core $[a,b]$~, there is nothing to prove.
\medbreak
ii) Consider the case $\partial\Omega=a$~.

Let
 $\gamma_1$~, \dots, $\gamma_r$ be as in definition \ref{thick_dfn}. Write $\gamma_{r+i}=\gamma_i\mun$ ($1\le i\le r$)~.
 Fix $u_0>b$
so close to $b$~, that $u_0$ belongs to the domains of $\gamma_1$~, \dots, $\gamma_{2r}$~.
Define by induction two sequences $u_n\in(b,u_0]$ and $1\le i(n)\le 2r$~,
such that $u_{n+1}:=\gamma_{i(n)}(u_{n})$ is the minimum of $\gamma_1(u_{n})$~,
\dots, $\gamma_{2r}(u_n)$~.
\medbreak
Claim: $\gamma:=\gamma_{i(0)}$ is fix-point free in $(b,u_0]$~.

 Indeed, consider any $t\in(b,u_0]$~. By property (3) of definition \ref{thick_dfn},
 there is no common fixed point for $\gamma_1$~, \dots, $\gamma_r$ in the interval
 $(b,u_0]$~. Thus, $(u_n)$ decreases to $b$~.
 By construction, the two following properties are obvious, for every $n\ge 1$ and every $0\le i\le 2r$~:
$$\gamma_{i(n)}\mun((u_{n+1},u_n])\subset(u_n,u_{n-1}]\ \ (A)$$
$$\gamma_i(b,u_{n}]\subset (b,u_{n-1}]\ \ (B)$$
 Let $N$ be the integer such that $t\in(u_{N+1},u_N]$~. 
By (A), the composite $\alpha:=\gamma_{i(1)}\mun\dots\gamma_{i(N)}\mun$ is defined at $t$,
and $\alpha(t)\in(u_1,u_0]$~. Consequently, $\gamma(\alpha(t))<\alpha(t)$~.
If $N=0$~, this means that $\gamma(t)<t$~, and we are done.
So, assume $N\ge 1$~.
Then, $\alpha$ is defined on the whole interval $(b,u_{N-1})$~, which contains $t$ and $\gamma(t)$ (by B).
 On the other hand, by (B), for every $k=0,\dots,N$~,
the composite
$$w_k:=\gamma_{i(1)}\mun\dots\gamma_{i(k)}\mun\gamma\gamma_{i(k+1)}\mun\dots\gamma_{i(N)}\mun$$ is defined at $t$~.
By property (2) of definition \ref{thick_dfn}~ , one has $w_0(t)=w_1(t)$~,\dots, $w_{N-1}(t)=w_N(t)$~. So,
$\alpha(\gamma(t))=\gamma(\alpha(t))<\alpha(t)$~. Applying $\alpha\mun$~, we get $\gamma(t)<t$~. The claim is proved.
\medbreak

iii) The case $\partial\Omega=b$ is of course symmetric to ii).
\medbreak
iv) In the remaining case $\partial\Omega=\emptyset$, we need a little more argument to get
a local transformation which is fix-point free \emph{in the same time} on the left of $a$, and on the right of $b$~.

Like in case iii), one makes a composite $\gamma$ of the generators $\gamma_i$'s, defined on some
neighborhood of $[a,b]$~, and such that $\gamma(t)<t$
for every $t>b$ in the domain of $\gamma$~. Symmetrically, one makes a composite $\delta$ of the generators $\gamma_i$'s, defined on some
neighborhood of $[a,b]$~, and such that $\delta(t)>t$
for every $t<a$ in the domain of $\delta$~. Clearly, by property (2), $\gamma\delta=\delta\gamma$ on some small neighborhood $[s_0,t_0]$ of $[a,b]$~.
 If $\gamma$ (resp. $\delta$) is fix-point free on $[s_0,a)$ (resp. on $(b,t_0]$)~, then $\gamma$ (resp. $\delta$) restricted to $(s_0,t_0)$ works.
So, disminushing the interval $[s_0,t_0]$ if necessary, we can assume that $\gamma(s_0)=s_0$ and $\delta(t_0)=t_0$~.
Then, $\gamma\delta$ restricted to $(s_0,t_0)$
 works. Indeed, let $t_n:=\gamma^n(t_0)$ and $s_n:=\delta^n(s_0)$~. Then, the sequence $t_n$ decreases to $b$~,
the sequence $s_n$ increases to $a$~, and  $\gamma(s_n)=s_n$~, and  $\delta(t_n)=t_n$~, and $\gamma\delta(t_n)=t_{n+1}$~, and
$\gamma\delta(s_n)=s_{n+1}$~. So, $\gamma\delta$ is fix-point free in $[s_0,a)$ and in $(b,t_0]$~.

\end{proof}
 
 Every hinge is easily realized:

 \begin{lem}\label{commutative_realization_lem}
 Let $(\Gamma,\Omega)$ be a hinge of rank $r\ge 1$.
 \begin{enumerate}
 \item If $r\le 2$, the hinge $(\Gamma,\Omega)$ is realized on $T^2\times[0,1]$, with $2-\sharp(\partial\Omega)$ transverse boundary components;
  \item For every $r$, the hinge $(\Gamma,\Omega)$ is realized on some compact $4$-manifold, without
  transverse boundary component.
 \end{enumerate}
 \end{lem}
 \begin{proof}
 The realization is much like in the particular
 case of the homothety pseudo-groups,
 seen above at paragraph  \ref{homothety_sbs}.
 Consider, to fix ideas, the case where $\partial\Omega=\emptyset$~.

(1)
Let us assume that $r= 2$,
 and let us realize $(\Gamma,\Omega)$ on $T^2\times[0,1]$~. Recall that $[a,b]\subset\Omega$ is the core
(definition \ref{thick_dfn}).
 
 The suspension of $\gamma_1$ and $\gamma_2$ over $T^2$
 provides,  in $T^2\times\Omega$~, a foliation $\F$ on some open neighborhood $U$ of $T^2\times[a,b]$~. By property (3)
of definition \ref{thick_dfn}, the local transformations $\gamma_1$ and $\gamma_2$ have
  no common fix point outside $[a,b]$~.
  Consequently, one has an embedding of $T^2\times[0,1]$ into $U$
 containing
 $T^2\times[a,b]$ in its interior, and meeting every leaf of $\F$~,
 and such that $T^2\times 0$ and $T^2\times 1$
 are embedded transversely to $\F$~. It is easily verified that $\F\vert(T^2\times[0,1])$
 realizes $(\Gamma,\Omega)$~.

 (2) Now, let $r$ be any integer $\ge 2$~. After lemma \ref{no_fix_point_lem}, we can assume that $\gamma_1$ is fix-point free
outside $[a,b]$~. Also, by a restriction of $\Omega$
which amounts to a Haefliger equivalence of the hinge pseudo-group, one arranges that $\Omega=Dom(\gamma_1)\cup Im(\gamma_1)$~.
Then, for each $2\le i\le r$~, just as in case (1), one
 realizes the pseudo-group $\Gamma_i:=<\gamma_1,\gamma_i>$
on $\Omega$ by a foliation $\F^3_i$ on $V:=T^2\times[0,1]$~.
  So,  $\Gamma_i$ is also realized by the pullback $\F^4_i$
  of $\F^3_i$ in $M_i:=\partial(V\times\D^2)
  \cong T^2\times\S^2$~. The foliation $\F^4_i$ contains a ``core'' $I$-bundle
  $B_i\cong T^3\times[a,b]$.
 For each $i=3,\dots,r$,  in $B_2$ (resp. $B_i$), we pick some embedded annulus $A_i:=C_ i\times[a,b]\subset B_2$
 (resp. $A'_i:=C'_ i\times[a,b]\subset B_i$), where $C_i$ (resp. $C'_i$) is a circle embedded in $T^3$ and
  parallel to the first circle factor. The foliation $\F^4_2\vert A_i$ (resp.  $\F^4_i\vert A'_i$) is the suspension of $\gamma_1\vert[a,b]$~;
 the holonomy of $\F^4_2$ (resp. $\F^4_i$) along $C_i\times a$
  (resp. $C'_i\times a$) is the germ of $\gamma_1$
  at $a$~; the holonomy of $\F^4_2$ (resp. $\F^4_i$) along $C_i\times b$
  (resp. $C'_i\times b$) is the germ of $\gamma_1$
  at $b$~.
 We arrange that $C_3$, \dots, $C_r$ are two by two disjoint in $T^3$~.
 The annulus $A_i$ (resp. $A'_i$) has in $M_2$ (resp. $M_i$) a small tubular neighborhood $N_i$ (resp. $N'_i$) $\cong\D^3\times\S^1$~, on
  the boundary of which $\F^4_2$ (resp. $\F^4_i$) traces a foliation composed of an $I$-bundle and two Reeb components,
  realizing $(<\gamma_1>,\Omega)$~.
  We cut from $M_2$~, \dots, $M_r$ the interiors of $N_3$~, \dots, $N_r$~,
  $N'_3$~, \dots, $N'_r$~. We paste every $\partial N_i$ with $\partial N'_i$~,
  such that $\F^4_2\vert\partial N_i$ matches $\F^4_i\vert\partial N'_i$~. We get
  a closed $4$-manifold with a foliation realizing $(\Gamma,\Omega)$~.
\end{proof}
   
\begin{pro} Let $(G,T)$ be a compactly generated
pseudo-group of dimension $1$, in which every essential orbit is commutative.

Then, after a Haefliger equivalence, $T$ splits as a disjoint union
$$T=T_0\sqcup\Omega_1\sqcup\dots\sqcup\Omega_\ell$$
such that
\begin{enumerate}\label{splitting_pro}
\item $T_0$ is a finite disjoint union of circles and compact intervals;
\item Each $\Omega_k$ ($1\le k\le\ell$) is the domain of a hinge $\Gamma_k\subset
G$ whose rank is at most the essential rank of $(G,T)$;
\item Each core $[a_k,b_k]\subset\Omega_k$ is $G$-saturated;
\item For every $t\in\Omega_k\setminus[a_k,b_k]$
 ($1\le k\le\ell$),
 the orbit $G(t)$
meets $T_0$.
\end{enumerate}
\end{pro}

We begin to prove proposition \ref{splitting_pro}.

    By a \emph{subpseudo-group} in $(G,T)$~,
  we mean a pseudo-group $(\Gamma,\Omega)$ such that $\Omega\subset T$ is topologically open,
  and that $\Gamma\subset G\vert\Omega$~.
 
\begin{dfn}\label{faithful_dfn}
   Let $(\Gamma,\Omega)\subset(G,T)$ be a hinge subpseudo-group.
   Let $[a,b]\subset\Omega$ be its core.
   
   (a) Assume that
    $(\Gamma,\Omega)$ is
   degenerate $(a=b)$. We call the hinge subpseudo-group
  \emph{faithful} if $G(a)$ is closed in $T$ and if $\Gamma_{a}=G_a$ (isotropy groups).

   (b) Assume that
    $(\Gamma,\Omega)$ is
   nondegenerate $(a\neq b)$. We call the hinge subpseudo-group
  \emph{faithful}
   if the $G$-saturation of $[a,b]$ is a component of $(G,T)$,
   if $\Gamma\vert[a,b]=G\vert[a,b]$,
   if $\Gamma_a=G_a$, and if $\Gamma_b=G_b$.
  \end{dfn}
  
  In case (b), 
   the $G$-saturation of $[a,b]$ is necessarily an $I$-bundle component of $G$.

  The notion of subpseudo-group is \emph{not} functorial with respect
   to the Haefliger equivalences. The following notion solves this difficulty.
   
\begin{dfn}\label{extension_dfn}   Given two pseudo-groups $(G,T)$ and $(\Gamma,\Omega)$, an \emph{extension}
   of $(G,T)$ by $(\Gamma,\Omega)$ is a pseudo-group $\bar G$ on the disjoint
   union $\bar T:=T\sqcup\Omega$ such that
   \begin{itemize}
   \item $T$ is exhaustive for $\bar G$;
   \item $G=\bar G\vert T$;
   \item $\Gamma\subset\bar G$.
   \end{itemize}
   \end{dfn}
   In particular, $(\bar G,\bar T)$ is Haefliger-equivalent to $(G,T)$, and  $(\Gamma,\Omega)$ is a subpseudo-group of $(\bar G,\bar T)$.
   
   For example, given two pseudo-groups $(G,T)$, $(\Gamma,\Omega)$
   and given a Haefliger equivalence $(\bar\Gamma,\Omega\sqcup \Omega_0)$
   between $(\Gamma,\Omega)$ and some subpseudo-group $(\Gamma_0,\Omega_0)\subset (G,T)$,
   one has a natural extension of  $(G,T)$ by $(\Gamma,\Omega)$: namely,
   $\bar G$ is the pseudo-group on $T\sqcup\Omega$ generated by
   $G\cup\bar\Gamma$.

   An extension $(\bar G,\bar T)$ of a pseudo-group $(G,T)$ by a hinge
   $(\Gamma,\Omega)$
   is called \emph{faithful} if $(\Gamma,\Omega)\subset(\bar G,\bar T)$ is faithful;
   \emph{essential} if its core is an essential orbit or
    an essential $I$-bundle in $(\bar G,\bar T)$.

  \begin{lem}\label{local_hinge_extension_lem}
  For every $t\in T$ such that $G(t)$ is essential,
  there is an essential faithful extension $(\bar G,\bar T)$ of $(G,T)$ by a hinge $(\Gamma,\Omega)$ s.t:
  \begin{itemize}
  \item The rank of the hinge $(\Gamma,\Omega)$ is at most the essential rank of $(G,T)$; 
  \item The core of $(\Gamma,\Omega)$
 meets $\bar G(t)$;
  \item The core of $(\Gamma,\Omega)$
 meets also every essential orbit of $\bar G$ close enough to
 $\bar G(t)$.
 \end{itemize}
\end{lem}
 \begin{proof} 
 First case: $G(t)$  is not contained in any
 $I$-bundle component of $(G,T)$ of rank $0$.
 
   In this case, we shall actually find a
   faithful hinge subpseudo-group in $(G,T)$ whose core
   meets $G(t)$ and every neighboring essential orbit.
   
 Let $r:=\rank(G_t)$ and choose $h_1$, \dots, $h_r\in G$
 such that their germs at $t$
 are a basis of $G_t$. Let $\Omega$ be a small interval containing $t$,
 topologically open in $T$, and contained in
 the intersection of the domains and of the images of  $h_1$, \dots, $h_r$.
 Put  $\gamma_i:=h_i\vert(\Omega\cap h_i\mun(\Omega))$ ($i=1,\dots, r$)
 and $\Gamma:=<\gamma_1,\dots,\gamma_r>$.
 For $\Omega$ small enough,  the properties (1) and (2) of definition \ref{thick_dfn} are fulfilled
for every small enough, $\Gamma$-invariant interval $[a,b]$~.

 First subcase: $G(t)$ is isolated (recall the vocabulary that follows lemma \ref{closed_orbits_lem} above).
 Put $a=b:=t$.
 For $\Omega$ small enough, by lemma 
 \ref{accumulation_lem}, $h_1,\dots, h_r$ have no common fix point in $\Omega$.
In consequence, for every $t'\in\Omega\setminus\{t\}$, there is an $i$ for which
one of the four following properties holds:
$t'\in Dom(\gamma_i)$ and $t<\gamma_i(t')<t'$, or $t'\in Dom(\gamma_i\mun)$ and  $t<\gamma_i\mun(t')<t'$, or $t'\in Dom(\gamma_i)$ and $t'<\gamma_i(t')<t$, or  $t'\in Dom(\gamma_i\mun)$ and $t'<\gamma_i\mun(t')<t$. The property (3) of definition \ref{thick_dfn}
follows.
 
Second subcase: $G(t)$ is not isolated from either side. In that subcase,
by lemma \ref{accumulation_lem},
 we can
shorten $\Omega$ to arrange that moreover none of the endpoints of
$\Omega$ is a fix point common to $h_1$, \dots, $h_r$.
Let $a$ and $b$ be the smallest and the largest fix points common to
 $h_1$, \dots, $h_r$
 in $\Omega$.
Then, $a<t<b$. For every $t'\in\Omega\setminus[a,b]$, there is an $i$ for which
 one of the four following properties holds:
$t'\in Dom(\gamma_i)$ and $b<\gamma_i(t')<t'$, or  $t'\in Dom(\gamma_i\mun)$ and $b<\gamma_i\mun(t')<t'$, or $t'\in Dom(\gamma_i)$ and $t'<\gamma_i(t')<a$, or  $t'\in Dom(\gamma_i\mun)$ and $t'<\gamma_i\mun(t')<a$. The property (3) of definition \ref{thick_dfn}
follows.

Third subcase: $G(t)$ is isolated from exactly one side. The argument is similar to the first two subcases.

Second case: $G(t)$ is contained in an $I$-bundle component $C\subset T$ of rank $0$.
That is, $C$ is a $1$-manifold, topologically closed in $T$,
 and $G\vert C$ is Haefliger-equivalent to the trivial pseudo-group
 on the interval $[0,1]\subset\R$. In other words, one has an orientation-preserving
 map $f:C\to[0,1]$ which is etale (that is, a local diffeomorphism); and the Haefliger equivalence is nothing but the pseudo-group
 on the disjoint union $C\sqcup[0,1]$
generated by the set of the local sections of $f$.
The boundary $\partial C$ is made of of two orbits
 $\partial_-C=G(t_0)$ and $\partial_+C=G(t_1)$.
 We can assume that $C$ is
 maximal among the $I$-bundle components of rank $0$.
  Assume also, to fix ideas, that $C$ is
interior to $T$ (the other cases being alike and simpler).
 Thus, the isotropy group of $G$ at $t_0$ (resp. $t_1$)
is nontrivial on the left (resp. right).

Pick
some small open interval $(u_0,v_0)\subset T$ containing $t_0$
and whose intersection with $C$ is $[t_0,v_0)$; and
 pick
some small open interval $(u_1,v_1)\subset T$ containing $t_1$
and whose intersection with $C$ is $(u_1,t_1]$. Take the intervals so small that
$f(v_0)<f(u_1)$.  One extends $f\vert[t_0,v_0)$ into a diffeomorphism $f_0$
from the interval $(u_0,v_0)$
onto the interval $(-\infty,f(v_0))$. The choice among the extensions is arbitrary.  Similarly, one extends $f\vert(u_1,t_1]$ into a diffeomorphism $f_1$ from the interval $(u_1,v_1)$
onto the interval $(f(u_1),+\infty)$. Let $T'$ be the disjoint union $T\sqcup\R$. Let $G'$ be the pseudo-group on $T'$ generated by
$G$, $f$, $f_0$, and $f_1$. Obviously, $T$ is exhaustive in $(G',T')$,
and $G=G'\vert T$, and $G'\vert[0,1]$ is the trivial pseudo-group on $[0,1]$, and
the orbit $G'(t)$ meets $[0,1]$ at $f(t)$.
 Let $r:=\max(\rank(G'_0),\rank(G'_1))$. One immediately
  makes $h_1$, \dots, $h_r\in G'\vert\R$ whose domains and images
 contain $[0,1]$, which are the identity on $(0,1)$, whose germs at $0$ generate
 $G'_0$, and whose germs at $1$ generate $G'_1$.
 Let $\Omega\subset\R$ be an open interval containing $[0,1]$, and contained in
 the intersection of the domains and of the images of  $h_1$, \dots, $h_r$.
  For $\Omega$ small enough,  the property (2) of definition \ref{thick_dfn} is fulfilled.
 By lemma \ref{accumulation_lem},
 we can moreover
shorten $\Omega$ to arrange that
 none of its endpoints is a fix point common to $h_1$, \dots, $h_r$.
 Put  $\gamma_i:=h_i\vert(\Omega\cap h_i\mun(\Omega))$ ($i=1,\dots, r$)
 and $\Gamma:=<\gamma_1,\dots,\gamma_r>$.
Let $a$ and $b$ be the smallest and the largest fix points common to
 $h_1$, \dots, $h_r$
 in $\Omega$. The property (3) of definition \ref{thick_dfn}
is fulfilled. The pseudo-group $(G'\vert(T\sqcup\Omega),T\sqcup\Omega)$
is a faithful extension of $(G,T)$ by the hinge $(\Gamma,\Omega)$.
 \end{proof}

 \begin{proof}[Proof of proposition \ref{splitting_pro}] The pseudo-group $(G,T)$ being cocompact, and the union of the essential leaves being topologically closed in $T$, one has a compact $K\subset T$
 whose $G$-saturation coincides with this union. By lemma \ref{local_hinge_extension_lem}, every point of $K$ has a
 neighborhood in $K$ whose orbits meet the core of the hinge after one essential, faithful hinge
 extension, whose rank is at most the essential rank of $(G,T)$. One extracts a finite subcover. There corresponds a finite sequence of
essential faithful extensions by hinges $(\Gamma_k,\Omega_k)$ ($1\le k\le\ell$),
 whose ranks are at most the essential rank of $(G,T)$.
 Let $(\bar G,\bar T)$ be the resulting global extension of $(G,T)$; let $[a_k,b_k]$
be the core of $(\Gamma_k,\Omega_k)$; and let $C_k\subset \bar T$ be the $\bar G$-saturation of
 $[a_k,b_k]$.  It is easy to arrange that   $C_1$,\dots, $C_\ell$ are two by two disjoint. A closed orbit of $\bar G$ is contained in  $C_1\cup\dots\cup C_\ell$ iff it is essential.
 Consequently, the pseudo-group $(\bar G\vert U=G\vert U,U)$ is taut, where
 $$U:=
T\setminus((C_1\cup\dots\cup C_\ell)\cap T)$$ Also, the topological
closure $\bar U$ of $U$ in $T$ being a component of $(G,T)$, the restricted pseudo-group $(G\vert\bar U,\bar U)$ is compactly
generated. By lemma \ref{complement_lem}, $(G\vert U,U)$
is also compactly generated.  By proposition \ref{taut_pro},
$(G\vert U,U)$ is Haefliger-equivalent to some pseudo-group $(G_0,T_0)$
 on a finite
disjoint union $T_0$ of compact intervals and circles. By the example that follows the
definition \ref{extension_dfn} above, we get an extension $(\tilde G,\tilde T)$
of $(\bar G,\bar T)$ by $(G_0,T_0)$. One has $\tilde T=\bar T\sqcup T_0$. Let
$$\tilde T':=T_0\sqcup\Omega_1\dots\sqcup\Omega_\ell\subset\tilde T$$
By construction, $\tilde T'$ is exhaustive in $(\tilde G,
\tilde T)$. We change $(G,T)$ for $(\tilde G\vert\tilde T',\tilde T')$.
The properties of proposition \ref{splitting_pro} are fulfilled.
 \end{proof}

\subsection{End of the proofs of theorems B and C}

Let, as before, $(G,T)$ be a compactly generated pseudo-group
of dimension $1$, in which every essential orbit is commutative.
Our task is to realize $(G,T)$, in dimension $3$ if possible, and $4$ if not.

Without loss of generality, $(G,T)$ is under the form described by
 proposition \ref{splitting_pro}.
We shall first realize separately
$(G\vert T_0,T_0)$, 
$(\Gamma_1,\Omega_1)$, \dots, $(\Gamma_\ell,\Omega_\ell)$;
and then perform some surgeries along some loops in the realizations, transverse
to the foliations. It is convenient to begin with somewhat introducing these loops into
the pseudo-group.

For each $k$, if
 $a_k\notin\partial_-\Omega_k$
 (resp.  $b_k\notin\partial_+\Omega_k$), write $\Omega_k^-$
 (resp. $\Omega_k^+$) the connected component of
$\Omega_k\setminus[a_k,b_k]$ on the left of $a_k$ (resp. on the right of $b_k$).

\begin{lem}\label{interval_lem}
a)  In case $a_k\notin\partial_-\Omega_k$, there exist in $\Omega_k^-$
   two
points $a'_k<a''_k<a_k$, and $\phi_k\in\Gamma_k$, such that

 i) The interval $(a'_k,a''_k)$ is exhaustive for $\Gamma_k\vert\Omega_k^-$;

 ii) $[a'_k,a''_k]\subset Dom(\phi_k)\cap Im(\phi_k)$;
 
  iii) $\phi_k(t)>t$ for every $t\in[a'_k,a''_k]$;
  
  iv) $\phi_k(a'_k)<a''_k$.

b) Symmetrically,  in case $b_k\notin\partial_+\Omega_k$,
 there exist in $\Omega_k^+$
 two
points $b_k<b'_k<b''_k$, and $\psi_k\in\Gamma_k$, such that

 i) The interval $(b'_k,b''_k)$ is exhaustive for $\Gamma_k\vert\Omega_k^+$;

 ii)  $[b'_k,b''_k]\subset Dom(\psi_k)\cap Im(\psi_k)$;
 
  iii) $\psi_k(t)>t$ for every $t\in[b'_k,b''_k]$;
  
  iv) $\psi_k(b'_k)<b''_k$.
\end{lem}
\begin{proof}[Proof of a)] Recall $\gamma_1$, \dots, $\gamma_r$ of definition \ref{thick_dfn}.
Choose
  $a'_k<a_k$, so close to $a_k$ that it belongs to the domain and to the image
 of $\gamma_i$,
  for every $1\le i\le r$.
 Let $\gamma_j^{\epsilon_j}(a'_k)$
 be the maximum of the values $\gamma_i(a'_k)$, $\gamma_i\mun(a'_k)$ ($1\le i\le r$).
  Put $\phi_k:=\gamma_j^{\epsilon_j}$.
Choose $a''_k$ in the interval $(\gamma_j^{\epsilon_j}(a'_k),a_k)$, so close to $\gamma_j^{\epsilon_j}(a'_k)$ that iii) holds. The properties i), ii) and iv) are obvious.
\end{proof}

For each $k=1$, \dots, $\ell$, it follows from ii), iii) and iv) that,
 in case $a_k\notin\partial_-\Omega_k$ (resp. $b_k\notin\partial_+\Omega_k$),
  the subpseudo-group of $(\Gamma_k,\Omega_k)$
generated by $\phi_k\vert(a'_k,a''_k)$ (resp. $\psi_k\vert(b'_k,b''_k)$) is Haefliger-equivalent to the trivial pseudo-group
on the circle. In case $\partial\Omega_k=\emptyset$ (resp. $\{a_k\}$)
(resp. $\{b_k\}$) (resp. $\{a_k,b_k\}$), by the example following the definition
\ref{extension_dfn},
 we get an extension $(\hat\Gamma_k,\hat\Omega_k)$
 of the hinge  $(\Gamma_k,\Omega_k)$
by the trivial pseudo-group on the disjoint union of two circles $S_k^-\sqcup S_k^+$
(resp. one circle $S_k^+$) (resp. one circle $S_k^-$) (resp. $\emptyset$).

In other words, we have an extension $(\hat G,\hat T)$ of $(G,T)$ by the trivial
pseudo-group  on the disjoint union $S$ of all the $S_k^\pm$'s ($1\le k\le\ell$).
In particular, $\hat T=T\sqcup S$.
Write $\hat T_0:=T_0\sqcup S\subset\hat T$ and $\hat G_0:=\hat G\vert\hat T_0$.
Also write $$A:=[a_1,b_1]\cup \dots \cup [a_\ell,b_\ell]$$

 \begin{lem}\label{generation_lem}
  The pseudo-group $\hat G$ on $\hat T$ is generated by $\hat
  G_0$, $\hat\Gamma_1$, \dots, and $\hat\Gamma_\ell$.
 \end{lem}

 \begin{proof}
 We have to verify that the germ $[g]_t$ of every $g\in\hat G$ at
 every $t\in Dom(g)$, is generated by $\hat G_0$ and the $\hat\Gamma_k$'s.

 If $t\in\Omega_k\setminus[a_k,b_k]$, then (lemma
 \ref{interval_lem}, i)) there is
 some $\gamma\in\Gamma_k$ such that $\gamma(t)\in (a'_k,a''_k)$ or 
 $\gamma(t)\in (b'_k,b''_k)$, and thus some $\hat\gamma\in\hat\Gamma_k$
 such that $\hat\gamma(t)\in S_k^\pm$.
 We are thus reduced to the case $t\in\hat T_0\cup A$. Symmetrically, one can assume also that $g(t)\in\hat T_0\cup A$.
 
 By proposition \ref{splitting_pro}, (3), either $t, g(t)\in\hat T_0$ (and thus $[g]_t\in
 \hat G_0$) or $t, g(t)\in [a_k,b_k]$ for some $1\le k\le\ell$. In that second case,
 the extension of $(G,T)$ by $(\Gamma_k,\Omega_k)$ being faithful,
 $g\in\Gamma_k$.
\end{proof}

 \begin{proof}[Proof of theorem B] We have to prove that (2) implies (1).
 Start from a pseudo-group $(\hat G,\hat T)$
 as in lemma \ref{generation_lem}, Haefliger-equivalent to $(G,T)$.
 
 On the one hand, the restriction of $\hat G$ to $\hat T\setminus Int(A)$, being a component
  of $(\hat G,\hat T)$,
 is compactly generated by lemma \ref{complement_lem}. Since $\hat T_0\subset \hat T\setminus Int(A)$
 is exhaustive, $(\hat G_0,\hat T_0)$ is compactly generated.
 This pseudo-group is also taut, $\hat T_0$ being a disjoint union of
 circles and compact intervals. By theorem A,  $(\hat G_0,\hat T_0)$ 
 is realized by a foliated compact $3$-manifold $(M_0,\F_0)$,
 without transverse boundary. More precisely, from the proof of theorem A,
 $\hat T_0$ is
 embedded into $M_0$ as an exhaustive transversal to $\F_0$,
 and $\hat G_0$ is the holonomy pseudo-group of
 $\F_0$ on $\hat T_0$. One takes off from $M_0$ a small open tubular
 neighborhood $N_0$ of $S$, such that $\F_0\vert\partial N_0$ is
 the trivial foliation by $2$-spheres.
 
 On the other hand, for each $k=1$, \dots, $\ell$, one realizes $(\Gamma_k,
 \Omega_k)$ by a foliation $\F_k$ on $M_k:=T^2\times[0,1]$ (lemma
 \ref{commutative_realization_lem}). Obviously, $\F_k$ admits transverse loops
 corresponding to $S_k^\pm$, in the sense that $\hat\Omega_k$
 embeds into $M_k$ as an exhaustive transversal to $\F_k$,
 and $\hat\Gamma_k$ is the holonomy pseudo-group of
 $\F_k$ on $\hat\Omega_k$.
  One takes off from $M_k$ a small open tubular
 neighborhood $N_k$ of $S_k^\pm$, such that $\F_k\vert\partial N_k$ is
 the trivial foliation by $2$-spheres.
 
 One pastes $\sqcup_{1\le k\le\ell}\partial N_k\cong\S^2\times S$ with $\partial N_0
 \cong\S^2\times S$,
 with respect to the identity of $S$. One gets a foliation $\F$ on $$M_0\cup_{\S^2\times S}(M_1
 \sqcup\dots\sqcup M_\ell)$$ whose holonomy on the exhaustive transversal
 $T$ coincides with $G$, by lemma \ref{generation_lem}.
 \end{proof}

  \begin{proof}[Proof of theorem C] The same as for (2) implies (1) in theorem B, 
  but instead of the foliated $3$-manifold $(M_0,\F_0)$, we use the foliated
  $4$-manifold $(M_0\times\S^1, pr_1^*(\F_0))$; and instead of $T^2\times[0,1]$,
  we use a $4$-dimensional realization of $(\Gamma_k,\Omega_k)$
  (lemma \ref{commutative_realization_lem}).
  \end{proof}

\bigskip
\noindent
Universit\'e de Bretagne Sud

\noindent
Universit\'e Europ\'eenne de Bretagne

\noindent
Laboratoire de Math\'ematiques de Bretagne Atlantique, UMR 6205

\noindent
Postal address: UBS, LMBA, B\^atiment Yves Coppens, Tohannic

\noindent B.P. 573

\noindent
F-56019 VANNES CEDEX, France

\noindent
Gael.Meigniez@univ-ubs.fr

\end{document}